\documentclass[reqno]{amsart}
\usepackage{amsmath,amssymb,colordvi}
\usepackage{multicol}
\usepackage{color,bbm}

 \newtheorem{theorem}{Theorem}[section]
 
 \newtheorem{lemma}[theorem]{Lemma}
 \newtheorem{proposition}[theorem]{Proposition}
 \newtheorem{example}[theorem]{Example}
 \newtheorem{definition}[theorem]{Definition}
 \newtheorem{remark}[theorem]{Remark}
 \newtheorem{assumption}[theorem]{Assumption}
 \numberwithin{equation}{section}
 \newcommand{\RR}{\mathbb{R}}
 \newcommand{\NN}{\mathbb{N}}
 \newcommand{\CC}{\mathbb{C}}

 \newcommand{\abs}[1]{\left\vert#1\right\vert}

 \newcommand{\norm}[1]{\left\Vert#1\right\Vert}

  \newcommand{\Om}{\Omega}
 \def\pOm{\partial\Omega}

 \def\bOm{\overline{\Om}}

 \usepackage[dvips,bottom=1.4in,right=1in,top=1in, left=1in]{geometry}

\def\Omc{\mathbb{R}^N\setminus\Omega}
\def\Omb{\mathbb{R}^N\setminus\overline{\Omega}}

\def\RR{{\mathbb{R}}}
\def\NN{{\mathbb{N}}}

\def\CC{{\mathbb{C}}}

\def\Om{\Omega}

\def\bOm{\overline{\Om}}
\def\pOm{\partial\Omega}
\begin{document}

\title[Approximate controllability]{Approximate  and mean approximate controllability properties for Hilfer time-fractional differential equations }

\author{Ernest Aragones}
\address{%
E. Aragones, University of Puerto Rico, Rio Piedras Campus\\
Faculty of Natural Sciences, Department of Mathematics \\
 17 University AVE. STE 1701, San Juan PR 00925-2537 (USA)}
\email{ernest.aragones@upr.edu}

\author{Valentin Keyantuo}
\address{%
V. Keyantuo, University of Puerto Rico, Rio Piedras Campus\\
Faculty of Natural Sciences, Department of Mathematics \\
 17 University AVE. STE 1701, San Juan PR 00925-2537 (USA)}
 \email{valentin.keyantuo1@upr.edu}

\author{Mahamadi Warma}
\address{M. Warma, Department of Mathematical Sciences,  George Mason University. Fairfax, VA 22030 (USA). }
\email{mwarma@gmu.edu}

\thanks{The work of the authors is partially supported by the Air Force Office of Scientific Research under Award NO: FA9550-18-1-0242}



\begin{abstract}
We study the approximate and mean approximate controllability properties of fractional partial differential equations associated with the so-called Hilfer type time-fractional derivative  and a non-negative selfadjoint operator $A_B$ with a compact resolvent on $L^2(\Omega)$, where $\Omega\subset\mathbb{R}^N$ ($N\ge 1$) is a bounded open set.  More precisely, we show that if $0\le\nu\le 1$, $0<\mu\le 1$ and $\Omega\subset\RR^N$ is a bounded open set, then the system 
$$\mathbb D_t^{\mu,\nu} u+A_Bu=f|_{\omega}\;\; \mbox{ in }\; \Omega\times (0,T),\,\,  (\mathbb I_t^{(1-\nu)(1-\mu)}u)(\cdot,0)=u_0  \mbox{ in }\;\Omega,$$  is approximately controllable for any $T>0$, $u_0\in L^2(\Om)$ and any non-empty open set $\omega\subset\Omega$. In addition, if the operator $A_B$ has the unique continuation property, then the system is also mean approximately controllable.
The operator $A_B$ can be the realization in $L^2(\Omega)$ of a symmetric, non-negative uniformly elliptic second order operator with Dirichlet or Robin boundary conditions, or the realization in $L^2(\Omega)$ of the fractional Laplace operator $(-\Delta)^s$ ($0<s<1$) with the Dirichlet exterior condition, $u=0$ in $\RR^N\setminus\Om$, or the nonlocal Robin exterior condition, $\mathcal N^su+\beta u=0$ in $\Omb$.
\end{abstract}

\maketitle

\section{Introduction}

Let $\Omega\subset\RR^N$ ($N\ge 1$) be a bounded open set with boundary $\pOm$. The main concern of the present paper is to study the controllability properties of a class of fractional (possible space-time) differential equations involving the so-called Hilfer time-fractional derivative. More precisely, we consider the following initial  value problem:
\begin{equation}\label{main-EQ}
\begin{cases}
\mathbb D_t^{\mu,\nu} u+A_Bu=f|_{\omega}\;\;&\mbox{ in }\; \Omega\times (0,T),\\
%Bu=0&\mbox{ on }\;\pOm\times (0,T),\\
(\mathbb I_t^{(1-\nu)(1-\mu)}u)(\cdot,0)=u_0 &\mbox{ in }\;\Omega,
\end{cases}
\end{equation}
where $T>0$ and $0\le\nu\le 1$, $0<\mu\le 1$ are real numbers, $\mathbb D_t^{\mu,\nu} u$ denotes the Hilfer time-fractional derivative of order $({\mu,\nu})$ of the function $u$ formally defined by 
\begin{align}\label{hfd}
\mathbb{D}_t^{\mu,\nu} u(t) := \mathbb I_t^{ \nu (1-\mu)}\frac{d}{dt}\left(\mathbb I_t^{(1-\nu)(1-\mu)}u\right)(t),\,\;\;t>0.
\end{align}
In \eqref{hfd}, for a real number $\alpha\ge 0$, $\mathbb I_t^\alpha$ denotes the Riemann-Liouville fractional integral of order $\alpha$
 (see \eqref{cfd} below for more details). 
 
 In \eqref{main-EQ},  the operator $A_B$ is a non-negative selfadjoint operator on $L^2(\Omega)$ with compact resolvent, $u=u(x,t)$ is the state of the system to be controlled and $f=f(x,t)$ is the control function which is localized in a non-empty open set $\omega\subset\Omega$.
 
 Let $f\in L^2(\omega\times (0,T))$ and $u_0\in C((0,T];L^2(\Omega))$ a solution of the system \eqref{main-EQ}. Then the set of reachable states is given by
 \begin{align*}
 \mathcal R(u_0,T):=\Big\{u(\cdot,T):\; f\in  L^2(\omega\times (0,T))\Big\}.
 \end{align*}
 
 We shall say the system \eqref{main-EQ} is null controllable if $0\in\mathcal R(u_0,T)$; exactly controllable if $\mathcal R(u_0,T)=L^2(\Omega)$; and approximately controllable if $\mathcal R(u_0,T)$ is dense in $L^2(\Omega)$. It is easy to see that exactly controllable implies null controllable which in turn implies approximately controllable. But the reserve implications are not true in general.
 
We will say that our system is mean approximately controllable if the set 
\begin{align*}
\left\{\mathbb I_t^{(1-\nu)(1-\mu)}u(\cdot,T):\; f\in L^2(\omega\times (0,T))\right\}
\end{align*}
is dense in $L^2(\Omega)$. This is a totally new notion of controllability and is different from the classical approximate controllability in the case $(1-\nu)(1-\mu)\ne 0$; otherwise the two notions coincide as we shall specify below.  It is also easy to see that  exactly controllable implies null controllable which also implies mean approximately controllable. But we do not know if there is any implication between mean approximately and approximately  controllable, except for the already observed fact that for $(1-\mu)(1-\nu)=0$ the two notions coincide.
 
 In our framework, the operator  $A_B$ can be a realization in $L^2(\Omega)$ of a symmetric and uniformly elliptic second order operator with bounded measurable coefficients subject to the Dirichlet or Robin type boundary conditions. Another example for the operator $A_B$ is a realization in $L^2(\Omega)$ of the fractional Laplace operator $(-\Delta)^s$ ($0<s<1$) with the Dirichlet exterior condition $u=0$ in $\RR^N\setminus\Om$, or the nonlocal Robin exterior condition, $\mathcal N^su+\beta u=0$ in $\Omb$, where $\beta\in L^1(\RR^N\setminus\Om)$ is a non-negative given function and $\mathcal N^su$ denotes the nonlocal normal derivative of the function $u$ (see \eqref{NLND} for the precise definition). We emphasize that with a small modification of our proofs the Neumann boundary condition (for second order elliptic operators) or the nonlocal Neumann exterior condition $\mathcal N^su=0$ in $\Omb$ (for the fractional Laplace operator) can be also included in our framework.
 
 %The case of the regional fractional Laplace operator $(-\Delta)_\Omega^s$ ($0<s<1$) with Dirichlet and Robin boundary conditions also enters in our framework.
 
When $\mu=1$, the system \eqref{main-EQ} is the well-known evolution equation of the first order which has been intensively studied. The heat and the Schr\"odinger equations are included in this framework. The null or/and the approximate controllability of such a system is well-known and has been investigated by several authors when $A_B$ is a realization in $L^2(\Omega)$ of a uniformly elliptic second order operator with various boundary conditions (Dirichlet, Neumann and Robin). We refer for instance  to the monographs \cite{Zua-el, Zua1} and their references for a complete overview.  
Instead, if $A_B$ is a realization in $L^2(\Omega)$ of the fractional Laplace operator $(-\Delta)^s$ ($0<s<1$) with Dirichlet, nonlocal Neumann or nonlocal Robin exterior conditions, little is known regarding the fractional heat equation.  In one space dimension ($N=1$), using some properties of the eigenfunctions, eigenvalues and the associated Ingham conditions, it has been shown that the fractional heat equation with the Dirichlet exterior condition is null controllable in any time $T>0$ if and only if $\frac 12<s<1$. See e.g. \cite{Bi-He,BWZ} for the interior control and \cite{ABPWZ,Wa-Za} for the exterior control, that is, when the control region $\omega$ is localized in $\Omb$. The case of the nonlocal Neumann and Robin exterior conditions remains open due to the lack of information on the associated eigenvalues and eigenfunctions.
In space dimension $N\ge 2$, the best possible controllability result available for the fractional heat and wave equations is the approximate controllability recently proved in \cite{Ke-Wa,LR-WA,War-AA,War-SIAM}. The case of the fractional Schr\"odinger equation in space dimension $N\ge 1$ has been studied in \cite{Bic} where the author has shown that the system is null controllable for large enough time $T$ if and only if $\frac 12\le s<1$. The main tool used is the fractional version of the Pohozaev identity established in \cite{R-S-P}. This is the only model that we can deal in the multi-dimensional setting.

%To the best of our knowledge, there is still no controllability results available when $A_B$ is a realization of the regional fractional Laplace operator with Dirichlet, Neumann or Robin boundary conditions.

%and the references therein.

If $\nu=1$ and $0<\mu<1,$ then $\mathbb D_t^{\mu}:=\mathbb D_t^{\mu,0}$ is the Caputo time fractional derivative of order $\mu.$  In this context, when $A_B=(-\Delta)^s$ ($0<s<1$) with the Dirichlet exterior condition, the exterior approximate controllability  properties have been investigated in \cite{War-SIAM}. Following the same ideas one can also derive some interior approximate controllability results.

%  The main problem is that the     solutions of the systems for $ \nu\ne 1$  do not behave like    the classical case, i.e. $\mu=1$.   
 
The interior approximate controllability in the case    $\mathbb D_t^{\mu}:=\mathbb D_t^{\mu,0}$ has been investigated in \cite{FY}  where the authors have shown that for a symmetric non-negative uniformly elliptic second order operator with the Dirichlet boundary condition, the corresponding fractional diffusion system is approximately controllable in any $T>0$, $\omega\subset\Omega$ an arbitrary non-empty open set and any $f\in C_0^\infty(\omega\times (0,T))$.

 Most recently, it has been shown in \cite{EZ} that for any $\mu>0$, ($\mu\not\in\NN$) the Caputo type system of order $\mu$ is not null controllable in any time $T>0$, that is, for example if $0<\mu<1$,  then there is no control function $f$ such that the solution $u$ of the associated system can rest at some time $T>0$. This also shows that such a system cannot be exactly controllable.
 We would like to emphasize that the proof of the non-null controllability given in \cite{EZ} also works for the system \eqref{main-EQ} when $\mu\ne 1$.
 %The boundary approximate controllability of the fractional diffusion equation for a symmetric non-negative uniformly elliptic operator with non-homogeneous Dirichlet boundary condition has also been studied in \cite{KF} with a positive result.

In the case $\nu=0$ and $0<\mu< 1,$ the above system becomes the fractional evolution equation with the Riemann-Liouville time-fractional derivative of order $\mu.$ Such equations have been intensively studied. The approximate controllability for equations in this class has been studied in \cite{LiuLi2015}. In the case of the Riemann-Liouville time-fractional derivatives, the initial condition is nonlocal as in \eqref{main-EQ} when $(1-\nu)(1-\mu)\ne 0$. 
%In effect, our equation interpolates between the ordinary derivative, the Caputo type derivative and the Riemann-Liouville derivative.

Motivated by  these results, we propose in this paper to investigate the case of the general fractional evolution equation as stated in \eqref{main-EQ} which includes all the above mentioned works. 
%We say that the system \eqref{main-EQ} is approximately controllable if for any    $u_1\in L^2(\Omega)$, $T>0$, and $\varepsilon>0$, there exists a control function $f$ such that the corresponding unique solution $u$ of the system \eqref{main-EQ} satisfies
%\begin{align*}
%\|u(\cdot,T)-u_1\|_{L^2(\Omega)} \le \varepsilon.
%\end{align*}

We summarize the novelties and the main results obtained in the present paper.

\begin{itemize}
\item As we have already mentioned, the system considered in the present paper is very general and  it includes almost all possible fractional order PDEs. In addition our framework includes not only the classical second order elliptic operators, but it also allows elliptic operators of fractional order like the fractional Laplace operator or the regional fractional Laplace operator (see e.g. \cite{War-N,War-In} for the definition of the regional fractional Laplacian).

\item The first main result (Theorem \ref{main-Theo}) states that  if $0\le \nu< 1$, $0<\mu< 1$ and $\omega\subset\Omega$ is an arbitrary non-empty open set, then the system \eqref{main-EQ} is approximately controllable in any time $T>0$. Given that such a system cannot be null controllable if $\mu\ne 1$ (by \cite{EZ}), the approximate controllability is the best possible result that can be expected in the study of the classical  controllability properties of the system \eqref{main-EQ}. 

\item Under the above hypothesis on $\omega$ and $0\le \nu\le 1$, $0<\mu\le 1$, our second main result (Theorem \ref{second-Theo}) shows that if in addition the operator $A_B$ has the unique continuation property (see \eqref{ucp}), then the system \eqref{main-EQ} is also mean approximately controllable in any time $T>0$. 

\item Finally, we show that the mean approximate controllability of  \eqref{main-EQ} is equivalent to the unique-continuation principle for solutions of the associated adjoint system \eqref{ACP-Dual}, that is,
\begin{align*}
(v\;\mbox{ solution of }\; \eqref{ACP-Dual},\;  v\big|_{\omega\times(0,T)}=0) \Longrightarrow \;v=0 \;\mbox{ in }\;\Omega\times (0,T).
\end{align*}
In addition  this seems not to be the case for the approximate controllability, unless $(1-\nu)(1-\mu)=0$, in which case mean approximate and approximate controllabilities are the same notions. 
\end{itemize}

%We also show that mean approximately controllable implies the approximate controllabiltiy; but we do not know if really these two notions are the same, that is, if the reserve implication also holds.
%
%\item Finally we have also proved that if the system is mean approximate controllable, then $u(\cdot,T)$ approximate $\mathbb I_t^{(1-\nu)(1-\mu)}u(\cdot,T)$ in the sense that for every $\varepsilon>0$, there is a control function $f\in L^2(\omega\times(0,T))$ such that the corresponding unique solution $u$ of \eqref{main-EQ} satisfies $\left \|\mathbb I_t^{(1-\nu)(1-\mu)}u(\cdot,T)- u(\cdot,I)\right\|_{L^2(\Omega)}\le\varepsilon$.

As we can observe in the system \eqref{main-EQ}, if $(1-\mu)(1-\nu)\ne 0$, then the initial condition is given in terms of the Riemann-Liouville fractional integral.
On the contrary, initial conditions for the Caputo derivatives (that is, when $\nu=1$ and $0<\mu<1$) are expressed
in terms of initial values of integer order derivatives. This allows for a
numerical treatment of initial value problems for differential equations of non
integer order independently of the chosen definition of the fractional derivative.
For this reason, many authors either resort to Caputo derivatives, or
use the Riemann-Liouville derivatives but avoid the problem of initial values
of fractional derivatives by treating only the case of zero initial conditions.  

The interesting paper \cite{Pod} has provided a series of examples from the field of viscoelasticity which demonstrate
that it is possible to attribute physical meaning to initial conditions
expressed in terms of Riemann-Liouville fractional derivatives (as in \eqref{main-EQ}),
and that it is possible to obtain initial values for such initial conditions by appropriate measurements or observations. The mentioned examples include: The Spring-pot model, which is a linear viscoelastic element whose behavior is intermediate between that of an elastic element 
and a viscous element; a stress relaxation or a general deformation; and an impulse response. For more details we refer to \cite{Pod} and the references therein.

 Fractional order operators have recently emerged as a modeling alternative in various branches of science and technology. In fact, in many situations, the fractional models reflect better the behavior of the system both in the deterministic and stochastic contexts. 
A number of stochastic models for explaining anomalous diffusion have been
introduced in the literature; among them we mention  the fractional Brownian motion; the continuous time random walk;  the L\'evy flights; the Schneider grey Brownian motion; and more generally, random walk models based on evolution equations of single and distributed fractional order in  space (see e.g. \cite{DS,GR,Man,Sch}).  In general, a fractional diffusion operator corresponds to a diverging jump length variance in the random walk. We refer to \cite{NPV,Val,GW-CPDE,GW-B,Pod} and the references therein for a complete analysis, the derivation and the applications of fractional order operators.

The rest of the paper is organized as follows. In Section \ref{main-res} we state the main results of the article and give some examples of operators that apply to our situation. Section \ref{preli} contains some intermediate results that are needed in the proofs of our main results. More precisely, we prove the existence and uniqueness  of weak solutions to the system \eqref{main-EQ} and its associated adjoint system, and give the representation of solutions in terms of series involving the Mittag-Leffler functions. The proofs of the main results are given in Section \ref{prof-ma-re}.

\section{Main results and examples}\label{main-res}

In this section we state the main results of the paper and give some examples. Let $\Omega\subset\RR^N$ ($N\ge 1$) be a bounded open set.
First, we introduce our assumption on the operator $A_B$.

%Let $A_B$ be the realization in $L^2(\Omega)$ of the operator $A_B$.

\begin{assumption}\label{asum-A}
 We assume that $A_B$ is a non-negative,  selfadjoint operator on $L^2(\Omega)$ with  domain $D(A_B)$, and that the embedding $D(A_B)\hookrightarrow L^2(\Omega)$ is compact, where $D(A_B)$ is endowed with the graph norm. 
\end{assumption}

It follows from Assumption \ref{asum-A}  that $A_B$ is given by a bilinear, symmetric, continuous, elliptic and closed form $\mathcal E_B$ with domain $D(\mathcal E_B):=V_{\frac 12}$ and $\mathcal E_B(u,u)\ge 0$ for every $u\in V_{\frac 12}$.  Moreover, we have that  $A_B$ has a compact resolvent, hence, its eigenvalues form a non-decreasing sequence of real numbers $0\le \lambda_1<
\lambda_2<\cdots<\lambda_n\cdots$ such that $\lim_{n\to\infty}\lambda_n=\infty$. We assume that the first eigenvalue $\lambda_1>0$. We denote by $(\varphi_n)_{n\in\NN}$ the orthonormal basis of normalized eigenfunctions associated with the eigenvalues $(\lambda_n)_{n\in\NN}$. Then $\varphi_n\in D(A_B)$ for every $n\in\NN$ and $(\varphi_n)_{n\in\NN}$ is total in $V_{\frac 12}$ and in $L^2(\Omega)$.

Throughout the remainder of the article, 
%for a real number $\gamma\ge 0$, we let $V_\gamma:=D(A_B^\gamma)$ be equipped with the norm
%\begin{align*}
%\|u\|_{V_\gamma}:=\|A_B^\gamma u\|_{L^2(\Omega)}.
%\end{align*}
%and 
without any mention we shall denote by $V_{-\frac 12}$ the dual of $V_{\frac 12}$ with respect to the pivot space $L^2(\Omega)$, so that we have the following continuous embeddings: $V_{\frac 12} \hookrightarrow L^2(\Omega)\hookrightarrow V_{-\frac 12}$. 
%Note that $(\varphi_n)_{n\in\NN}$ is also total in $V_{\gamma}$ if $0\le\gamma\le 1$. 
Moreover, $(\cdot,\cdot)_{L^2(\Om)}$ will designate the scalar product in $L^2(\Omega)$ and $\langle \cdot,\cdot\rangle_{V_{-\frac 12},V_{\frac 12}}$ will denote the duality pairing between $V_{-\frac 12}$ and $V_{\frac 12}$.

\subsection{Main results}

We first introduce the notion of weak solutions to the system \eqref{main-EQ}. 

\begin{definition}\label{def-strong-sol}
Let $u_0\in L^2(\Omega)$ and $f$  a given function.
A function $u$ is said to be a weak solution of the system \eqref{main-EQ}, if for every $T>0$,   the following properties hold:
\begin{itemize}
\item Regularity and initial condition:
\begin{equation}\label{regu}
\begin{cases}
u\in C((0,T];V_{\frac 12}), \\
\mathbb I_t^{(1-\nu)(1-\mu)}u\in C([0,T];L^2(\Omega)),\\
\mathbb D_t^{\mu,\nu} u\in C((0,T];V_{-\frac 12}),
\end{cases}
\end{equation}
and $\mathbb I_t^{(1-\nu)(1-\mu)}u(\cdot,0)=u_0$.

\item Variational  identity: For every $\varphi\in  V_{\frac 12}$ and a.e. $t\in (0,T)$, we have
\begin{align}\label{Var-I}
\langle \mathbb D_t^{\mu,\nu} u(\cdot,t),\varphi\rangle_{V_{-\frac 12},V_{\frac 12}}+\mathcal E_B(u(\cdot,t),\varphi)-(f(\cdot,t),\varphi)_{L^2(\Om)}=0.
\end{align}
\end{itemize}
\end{definition}

Now, we introduce our two notions of approximately controllable.   

\begin{definition} 
The system \eqref{main-EQ} is said to be approximately controllable in time $T>0$, if for every $u_0$, $u_1\in L^2(\Om)$ and $\varepsilon>0$, there exists a control function $f\in L^2((0,T);L^2(\omega))=L^2(\omega\times(0,T))$ such that the weak  solution $u$ of \eqref{main-EQ} satisfies
\begin{align}\label{e23}
\norm{u(\cdot,T)-u_1}_{L^2(\Om)}\le \varepsilon.
\end{align}
\end{definition}

\begin{definition} 
The system \eqref{main-EQ} will be said to be mean approximately controllable in time $T>0$, 
If for any $u_0$, $u_1\in L^2(\Om)$ and $\varepsilon>0$, there exists a control function $f\in L^2(\omega\times(0,T))$ such that the weak solution $u$ of \eqref{main-EQ} satisfies
\begin{align}\label{e24}
\norm{\mathbb I_t^{(1-\nu)(1-\mu)}u(\cdot,T)-u_1}\le\varepsilon.
\end{align}
%then \ref{main-EQ} is said to be approximate mean controllability.
\end{definition}

\begin{remark}
{\em We observe that if $(1-\nu)(1-\mu)=0$, then $\mathbb I_t^{(1-\nu)(1-\mu)}u(\cdot,T)=u(\cdot,T)$, and hence, approximate and mean approximate controllability coincide.}
\end{remark}
%First, we show that the approximate controllability of the system \eqref{main-EQ} is equivalent to the unique continuation property for the dual system.

%\begin{theorem}\label{theo024}
%The following assertions are equivalent.
%\begin{enumerate}
%\item The system \eqref{main-EQ} is approximately controllable for any $T>0$, $\omega\subset\Omega$ an open set and for every $f\in L^\infty((0,T), L^2(\omega))$.

%The following unique continuation principle holds for the solution of \eqref{ACP-Dual}:
%\begin{equation}
%v(x,t)=0\;\mbox{ in }\;\omega\times (0,T)\;\Longrightarrow\; (v_0,v_1)=(0,0).
%\end{equation}
%\end{enumerate}
%\end{theorem}
Our main results are the following theorems.

\begin{theorem}\label{main-Theo}
Let $0\le \nu<1$, $0<\mu< 1$ and $\omega\subset\Omega$ an arbitrary non-empty open set. Then, the system \eqref{main-EQ} is approximately controllable at any time $T>0$.
%, and an arbitrary open set $\omega\subset\Omega$. That is,
%\begin{align*}
%\overline{\{u(\cdot,T):\; f\in L^2((0,T);L^2(\omega)\}}^{ L^2(\Omega)}= L^2(\Omega),
%\end{align*}
%where $u$ is the unique weak solution of the system \eqref{main-EQ}.
\end{theorem}

\begin{theorem}\label{second-Theo}
Let $0\le \nu\le 1$, $0<\mu\le 1$ and $\omega\subset\Omega$ an arbitrary non-empty open set. Assume that  the operator $A_B$ has the unique continuation property in the sense that, 
\begin{align}\label{ucp}
\mbox{ if }\lambda>0,\;\; \varphi\in D(A_B),\;\; A_B\varphi=\lambda\varphi\;\mbox{ and } \varphi=0\mbox{ in }\omega, \mbox{ then } \varphi=0 \mbox{ in }\Omega.
 \end{align}
Then, the system \eqref{main-EQ} is mean approximately controllable.
\end{theorem}

%We obtain the following as a direct consequence of the above two results.
%
%\begin{corollary}\label{cor}
%Let $0<\nu,\,\mu<1$ and $\omega\subset\Omega$ an arbitrary non-empty open set. Assume that the system \eqref{main-EQ} is mean approximately controllable. Then the following assertions hold.
%
%\begin{enumerate}
%\item The system  \eqref{main-EQ} is approximately controllable.
%
%\item For every $\varepsilon>0$, there exists a control function $f\in L^2(\omega\times(0,T))$ such that the weak solution $u$ of \eqref{main-EQ} satisfies
%\begin{equation}\label{u-m}
%\left\|u(\cdot,T)-\mathbb I_t^{(1-\nu)(1-\mu)}u(\cdot,T)\right\|_{L^2(\Omega)}\le \varepsilon.
%\end{equation}
%\end{enumerate}
%\end{corollary}

%We conclude the paper with the following result which is a direct consequence of Theorems \ref{main-Theo} and \eqref{second-Theo}.
%
%
%
%\begin{proposition}
%Assume that the operator $A_B$ has the unique continuation property in the sense of \eqref{ucp}. Then for every $\varepsilon>0$, there exists a control function $f\in L^2(\omega\times(0,T))$ such that the unique solution $u$ of the system \eqref{main-EQ} satisfies
%\begin{equation}\label{u-m}
%\left\|u(\cdot,T)-\mathbb I_t^{(1-\nu)(1-\mu)}u(\cdot,T)\right\|_{L^2(\Omega)}\le \varepsilon.
%\end{equation}
%\end{proposition}

\subsection{Some examples of operators} We conclude this section by giving some examples of operators that satisfy our assumptions.

\begin{example}[\bf Second orders elliptic operators]
{\em We consider the operator $A$ given formally by
\begin{align*}
Au(x):=-\sum_{i,j=1}^N\frac{\partial}{\partial x_j}\left(a_{i,j}(x)\frac{\partial u}{\partial x_i}(x)\right) +b(x)u(x),\;\;x\in\Omega,
\end{align*}
where the real-valued coefficients satisfy the following conditions:
\begin{align*}
a_{ij}(x)=a_{ji}(x),\;\;a_{ij}\in W^{1,\infty}(\Omega),\;1\le i,j\le N,\; b\in L^\infty(\Omega),\;b(x)\ge 0,\; x\in\Omega,
\end{align*}
and there is a constant $\rho>0$ such that the following ellipticity condition holds: For a.e. $x\in\Omega$ and all $\xi\in\RR^N$ we have
\begin{align*}
\sum_{i,j=1}^Na_{i,j}(x)\xi_i\xi_j\ge \rho |\xi|^2.
\end{align*}

\begin{enumerate}
\item {\bf The Dirichlet boundary condition}.  Let $\Omega\subset\RR^N$ be an arbitrary bounded open set.
Let $A_D$ be the   selfadjoint operator on $L^2(\Omega)$ associated with the closed, bilinear and symmetric form $\mathcal E_D:W_0^{1,2}(\Omega)\times W_0^{1,2}(\Omega)\to \RR$ given by
\begin{align*}
\mathcal E_D(u,v):=\sum_{i,j=1}^N\int_{\Omega}a_{ij}\frac{\partial u}{\partial x_i}\frac{\partial v}{\partial x_j}\;dx+\int_{\Omega}buv\;dx,\;\;\;u,v\in W_0^{1,2}(\Omega).
\end{align*}
More precisely,
\begin{equation*}
\begin{cases}
D(A_D):=\left\{u\in W_0^{1,2}(\Omega):\;\exists\;w\in L^2(\Omega):\;\mathcal E_D(u,v)=(w,v)_{L^2(\Om)}\;\forall\;v\in W_0^{1,2}(\Omega)\right\},\\
A_Du=w.
\end{cases}
\end{equation*}
Then, $A_D$ is the realization of $A$ in $L^2(\Om)$ with the Dirichlet boundary condition $Bu:=u|_{\pOm}=0$ on $\pOm$, and it satisfies all our assumptions. In particular $A_D$ has the unique continuation property in the sense of \eqref{ucp}.

\item {\bf The Robin boundary condition}. Assume that $\Omega\subset\RR^N$ is a bounded open set with a Lipschitz-continuous boundary.
Let $\beta\in L^\infty(\pOm)$ satisfy $\beta(x)\ge \beta_0>0$ a.e. on $\pOm$ for some constant $\beta_0$. Let $A_R$ be the  selfadjoint operator on $L^2(\Omega)$ associated with the closed, bilinear and symmetric form  $\mathcal E_R:W^{1,2}(\Omega)\times W^{1,2}(\Omega)\to \RR$ given by:
\begin{align*}
\mathcal E_R(u,v):=\sum_{i,j=1}^N\int_{\Omega}a_{ij}\frac{\partial u}{\partial x_i}\frac{\partial v}{\partial x_j}\;dx+\int_{\Omega}buv\;dx+\int_{\pOm}\beta uv\;d\sigma.
\end{align*}
As for the Dirichlet boundary condition, we have that
\begin{equation*}
\begin{cases}
D(A_R):=\left\{u\in W^{1,2}(\Omega):\;\exists\;w\in L^2(\Omega):\;\mathcal E_R(u,v)=(w,v)_{L^2(\Om)}\;\forall\;v\in W^{1,2}(\Omega)\right\},\\
A_Ru=w.
\end{cases}
\end{equation*}
Then, $A_R$ is the realization of $A$ in $L^2(\Omega)$ with the Robin boundary condition $Bu:=\partial_{\nu_A}u+\beta u=0$ on $\pOm$. Here,
\begin{align*}
\partial_{\nu_A}u(x):=\sum_{i,j=1}^Na_{ij}(x)\frac{\partial u(x)}{\partial x_i}\nu_j(x).
\end{align*}
\end{enumerate}
The operator $A_R$ satisfies all our assumptions and it also enjoys the unique continuation property. For more details we refer to \cite{AW2,AW1,Dan,War-T} and the references therein.}
\end{example}

Before  giving  some examples involving the fractional Laplace operator, we need to introduce the fractional order Sobolev spaces   needed for  a rigorous definition of the fractional Laplace operator.
Let $\Om \subset\RR^N$ ($N \ge 1$) be an arbitrary bounded open set. For $0 < s < 1$ we let
 \begin{align}\label{SB}
    W^{s,2}(\Om) := \left\{ u \in L^2(\Om) :\; 
            \int_\Om\int_\Om \frac{|u(x)-u(y)|^2}{|x-y|^{N+2s}}\;dxdy < \infty \right\} 
 \end{align}
and we endow it with the norm 
 \begin{align*}
    \|u\|_{W^{s,2}(\Om)} := \left(\int_\Om |u|^2\;dx 
        + \int_\Om\int_\Om  \frac{|u(x)-u(y)|^2}{|x-y|^{N+2s}}\;dxdy \right)^{\frac12}.
 \end{align*}
We let 
 \begin{align*}
    W^{s,2}_0(\overline\Om) := \left\{ u \in W^{s,2}(\RR^N) :\; u = 0 \mbox{ in } 
            \RR^n\setminus\Om \right\}=\Big\{u\in W^{s,2}(\RR^N):\; \mbox{supp}(u)\subset\bOm\Big\}.
 \end{align*}
% where $W^{s,2}(\RR^N)$ is defined as in \eqref{SB} with $\Omega$ replaced by $\RR^N$. We have used $\bOm$ in the definition of $W^{s,2}_0(\overline\Om)$ in the spirit to avoid a confusion with the well-known space $W_0^{s,2}(\Omega):=\overline{\mathcal D(\Omega)}^{W^{s,2}(\Omega)}$.
 We set
 \begin{align*}
 \widetilde  W^{s,2}_0(\Om):=\Big\{u|_{\Omega}:u\in  W^{s,2}_0(\overline\Om)\Big\}.
 \end{align*}
%We also define the  local fractional order Sobolev space
% \begin{equation}\label{eq:Ws2loc}
%    W^{s,2}_{\rm loc}(\Omb) := \left\{ u \in L_{\rm loc}^2(\Omb):\; u\varphi \in W^{s,2}(\Omb) 
%         \; \forall \ \varphi \in \mathcal{D}(\Omb) \right\}.  
% \end{equation} 
 
% \begin{remark}\label{rem-sob}
% {\em 
%It is well-known that the following continuous embeddings hold:
%\begin{align}\label{sob-em}
% \widetilde  W^{s,2}_0(\Om)\hookrightarrow
%\begin{cases}
%L^{\frac{2N}{N-2s}}(\Omega)\;\;&\mbox{ if }\; N>2s,\\
%L^p(\Omega)\;\qquad \qquad\forall\;  p\in [1,\infty) \;&\mbox{ if }\; N=2s,\\
%C^{0,1-\frac{N}{2s}}(\bOm)&\mbox{ if }\; N<2s.
%\end{cases}
%\end{align}
%In addition to (\ref{sob-em}), we know that the embedding $\widetilde  W^{s,2}_0(\Om)\hookrightarrow L^2(\Omega)$ is compact.  If $\Omega$ has a Lipschitz continuous boundary, then \eqref{sob-em} also holds with $\widetilde  W^{s,2}_0(\Om)$ replaced with $W^{s,2}(\Omega)$. We refer to \cite[Chapter 1]{Gris}  for the proof of the above results (see also \cite{NPV} and the references therein).
%}
%\end{remark}

Next, let $\beta\in L^1(\Omc)$ be fixed and define the fractional order Sobolev type space
\begin{align*}
W_{\beta,\Omega}^{s,2}:=\Big\{u:\RR^N\to\RR\;\mbox{ measurable}:\;\|u\|_{W_{\beta,\Omega}^{s,2}}<\infty\Big\},
\end{align*}
where 
\begin{align}\label{norm-RV}
\|u\|_{W_{\beta,\Omega}^{s,2}}:=\left(\int_{\Omega}|u|^2\;dx+\int_{\Omc}|u|^2|\beta|\;dx+\int\int_{\RR^{2N}\setminus(\RR^N\setminus\Omega)^2}\frac{|u(x)-u(y)|^2}{|x-y|^{N+2s}}\;dxdy\right)^{\frac 12},
\end{align}
and
 \begin{align*}
\RR^{2N}\setminus(\RR^N\setminus\Om)^2
       = (\Om\times\Om)\cup(\Om\times(\RR^N\setminus\Om))\cup((\RR^N\setminus\Om)\times\Om).
\end{align*}
The space $W_{\beta,\Omega}^{s,2}$ has been introduced in \cite{SDipierro_XRosOton_EValdinoci_2017a} to study the nonlocal Neumann problem for $(-\Delta)^s$. If $\beta=0$, then we shall denote $W_{0,\Omega}^{s,2}=W_{\Omega}^{s,2}$. Then $W_{\beta,\Omega}^{s,2}\hookrightarrow W_{\Omega}^{s,2}$ as this is obvious from the above definitions.
It has been shown in \cite[Proposition 3.1]{SDipierro_XRosOton_EValdinoci_2017a} that  $W_{\beta,\Omega}^{s,2}$ endowed with the norm \eqref{norm-RV} is a Hilbert space.

For more information on fractional order Sobolev spaces we refer to \cite{NPV,SDipierro_XRosOton_EValdinoci_2017a,Val,Gris,War-N} and the corresponding references.
 
To introduce the fractional Laplace operator  we set 
\begin{equation*}
\mathcal{L}_s^{1}(\RR^N):=\left\{u:\RR^N\rightarrow
\mathbb{R}\;\mbox{ measurable}: \;\int_{\RR^N}\frac{|u(x)|}{(1+|x|)^{N+2s}}%
\;dx<\infty \right\}.
\end{equation*}
For $u\in \mathcal{L}_s^{1}(\RR^N)$ and $ 
\varepsilon >0$ we let
\begin{equation*}
(-\Delta )_{\varepsilon }^{s}u(x):=C_{N,s}\int_{\left\{y\in \RR^N:|y-x|>\varepsilon \right\}}
\frac{u(x)-u(y)}{|x-y|^{N+2s}}dy,\;\;x\in\RR^N,
\end{equation*}%
where the normalization constant $C_{N,s}$ is given by
\begin{equation}\label{CN}
C_{N,s}:=\frac{s2^{2s}\Gamma\left(\frac{2s+N}{2}\right)}{\pi^{\frac
N2}\Gamma(1-s)},
\end{equation}%
and $\Gamma $ is the usual Euler Gamma function. The {\bf fractional Laplace operator} 
$(-\Delta )^{s}$ is defined for $u\in \mathcal{L}_s^{1}(\RR^N)$  by the formula:
\begin{align}
(-\Delta )^{s}u(x):=C_{N,s}\mbox{P.V.}\int_{\RR^N}\frac{u(x)-u(y)}{|x-y|^{N+2s}}dy 
=\lim_{\varepsilon \downarrow 0}(-\Delta )_{\varepsilon
}^{s}u(x),\;\;x\in\RR^N,\label{eq11}
\end{align}%
provided that the limit exists for a.e. $x\in\RR^N$. We have that $\mathcal{L}_s^{1}(\RR^N)$ is the right space for which $v:=(-\Delta )_{\varepsilon }^{s}u$ exists for every $\varepsilon>0$ and $v$ being also continuous at the continuity points of $u$.  

For more information on the fractional Laplace operator we refer to \cite{SV1,SV2,War-N,War-In} and their references.

%\begin{remark}\label{rem-Diri}
%{\em We make the following observation.
%Let $f\in L^2(\Omega)$ and $u\in W_0^{s,2}(\bOm)$ a weak solution of \eqref{DiPr}. 
%We do not know if $u$ is a strong solution of \eqref{DiPr} in the sense that $(-\Delta)^su$ (as given in \eqref{eq11}) is well defined almost everywhere, $(-\Delta)^su\in L^2(\Omega)$ and $(-\Delta)^su=f$ a.e. in $\Omega$.
%However, we have the following inner regularity properties of solutions to \eqref{DiPr}. 
%By \cite[Theorem 1.3]{BWZ1} weak solutions of \eqref{DiPr} belong to $W_{\rm loc}^{2s,2}(\Omega)$. This maximal inner regularity result can be a starting point to investigate if weak and strong solutions coincide.
%}
%\end{remark}

%The following result characterizes the realization in $L^2(\Om)$ of  $(-\Delta)^s$ with the zero Dirichlet exterior condition $u_D=0$ in $\Omc$.

\begin{example}[\bf The fractional Laplacian with Dirichlet exterior condition]
{\em Firstly, we consider the Dirichlet problem for $(-\Delta)^s$, that is, the elliptic equation
\begin{equation}\label{DiPr}
(-\Delta)^su=f\;\;\mbox{ in }\;\Omega,\;\;\; u=0\;\mbox{ in }\;\Omc.
\end{equation}
Let $f\in L^2(\Omega)$. A function $u\in W_0^{s,2}(\bOm)$ is said to be a weak solution of \eqref{DiPr}, if 
\begin{align}\label{DiPrWe}
\mathcal E(u,v):=\frac{C_{N,s}}{2}\int_{\RR^N} \int_{\RR^N} \frac{(u(x)-u(y))(v(x)-v(y))}{\vert x-y\vert^{N+2s} } dy \;dx=\int_{\Omega}fv\;dx,
\end{align}
for every $v\in W_0^{s,2}(\bOm)$. Using the classical Lax-Milgram lemma, it is straightforward to show the existence and uniqueness of weak solutions to the Dirichlet problem \eqref{DiPr}.

Secondly, for a function \(u \in L^2(\Omega)\) we define its extension $u_D$ as follows:
\begin{align}\label{uD}
    u_D(x):=\begin{cases} u(x) &\text{ if } x\in \Omega, \\
    0 &\text{ if } x\in\RR^N\setminus\Omega.
    \end{cases}
\end{align}

Let 
\begin{align}\label{dom-DC}
    D(\mathcal E_D):=\Big\{u \in L^2(\Omega): u_D\in  W^{s,2}_0(\overline\Om)\Big\}=\widetilde W_0^{s,2}(\Om),
\end{align}
and $\mathcal E_D:D(\mathcal E_D) \times D(\mathcal E_D) \rightarrow \mathbb{R}\) the form given by
\begin{align*}
   \mathcal E_D(u,v):=\mathcal{E}(u_D,v_D),
\end{align*}
where $\mathcal E$ is given in \eqref{DiPrWe}.
Then, $\mathcal E_D$ is a densely defined, symmetric and closed bilinear form in $L^2(\Omega)$. The selfadjoint operator $(-\Delta)_D^s$ on $L^2(\Om)$ associated with $\mathcal E_D$ is given by  
\begin{equation}\label{Op-Dir}
\begin{cases}
 D((-\Delta)_D^s):=\Big\{u \in \widetilde W_0^{s,2}(\Om):\;\exists\; f\in L^2(\Om)\mbox{ such that } u_D \text{ is a weak solution of }  \eqref{DiPr}\\
 \hfill \text{ with right hand side } f \Big\},\\
 (-\Delta)_D^su:=f.
 \end{cases}
\end{equation}
By \cite{SV2} (see also \cite{Cl-Wa,SV1,War-N}) the operator $(-\Delta)_D^s$ has  a compact resolvent and its eigenvalues form a non-decreasing sequence of real numbers    $0< \lambda_1\le
\lambda_2\le \cdots\le \lambda_n\le\cdots$  satisfying $\lim_{n\to\infty}\lambda_n=\infty$.  All the eigenvalues have  finite geometric multiplicity. It also follows from \cite[Theorem 1.4]{FF} that  $(-\Delta)_D^s$ satisfies the unique continuation property in the sense of \eqref{ucp}.
}
\end{example}

\begin{example}[\bf The fractional Laplacian with nonlocal Robin exterior condition]
{\em Let $\Omega\subset\RR^N$ be a bounded open set with a Lipschitz continuous boundary.
For $u \in W_\Omega^{s,2}$ we define the nonlocal normal derivative $\mathcal{N}^su$ of $u$ as follows:
 \begin{align}\label{NLND}
    \mathcal{N}^s u(x) := C_{N,s} \int_\Om \frac{u(x)-u(y)}{|x-y|^{N+2s}}\;dy, 
            \quad x \in \RR^N \setminus \overline\Om,
 \end{align}
 where $C_{N,s}$ is the constant given in \eqref{CN}.
Clearly, $\mathcal{N}^s$ is a nonlocal operator and is well defined on $ W_\Omega^{s,2}$.

 Let $f\in L^2(\Omega)$, $\beta\in L^1(\Omc)$ a non-negative function and consider the following Robin problem:
\begin{equation}\label{RoPr}
(-\Delta)^su=f\;\;\mbox{ in }\;\Omega,\;\;\; \mathcal N^su+\beta u=0\;\mbox{ in }\;\Omb.
\end{equation}
By a weak solution to \eqref{RoPr} we mean a function $u\in W_{\beta,\Omega}^{s,2}$ such that
\begin{align*} %\label{RoPrWe}
\frac{C_{N,s}}{2}\int\int_{\RR^{2N}\setminus(\RR^N\setminus\Omega)^2}\frac{(u(x)-u(y))(v(x)-v(y))}{|x-y|^{N+2s}}\;dxdy+\int_{\Omc}\beta uv\;dx=\int_{\Omega}fv\;dx
\end{align*}
for every $v\in W_{\beta,\Omega}^{s,2}$. Here also the existence and uniqueness of weak solutions is easy to prove.

For a function $u\in L^2(\Om)$ we define its extension \(u_R\) as follows:
\begin{align*}
    u_R(x):=\begin{cases} u(x) &\text{ if } x \in \Omega, \\
 \displaystyle   \frac{C_{N,s}}{C_{N,s}\rho(x)+\beta(x)}\int_\Omega \frac{u(y)}{\vert x-y\vert^{N+2s}} dy &\text{ if } x\in\Omb,
    \end{cases}
\end{align*}
where
\begin{align*}%\label{rho}
    \rho(x):=\int_\Omega \frac{1}{\vert x-y\vert^{N+2s}} dy,\;\;x\in\Omb.
\end{align*}
Since $\partial\Omega$ is a null set (with respect to the $N$-dimensional Lebesgue measure), we have  that $u_R$ is well defined for every $u\in L^2(\Omega)$. In addition,  $u_R$ satisfies the following Robin exterior condition (see e.g. \cite{Cl-Wa}):
\begin{align}\label{NLRC}
\mathcal N^su_R+\beta u_R=0\;\mbox{ in }\;\Omb.
\end{align}
Let 
\begin{align*}%\label{form-Rob1}
    D(\mathcal E_R):=\Big\{ u \in L^2(\Om):\;u_R \in W^{s,2}_{\beta,\Omega}\Big \},
\end{align*}
and  the bilinear form $\mathcal E_R :D(\mathcal E_R)\times D(\mathcal E_R)\to\RR$ be given by
\begin{align*}
   \mathcal E_R(u,v):=\frac{C_{N,s}}{2}\int\int_{\RR^{2N}\setminus(\RR^N\setminus\Omega)^2}\frac{(u_R(x)-u_R(y))(v_R(x)-v_R(y))}{|x-y|^{N+2s}}\;dxdy+\int_{\Omc} \beta u_R v_R\;dx.
\end{align*}
%where here we have set
%\begin{align*}
%\mathcal{E}(u_R,v_R):=\frac{C_{N,s}}{2}\int\int_{\RR^{2N}\setminus(\RR^N\setminus\Omega)^2}\frac{(u_R(x)-u_R(y))(v_R(x)-v_R(y))}{|x-y|^{n+2s}}\;dxdy.
%\end{align*}
Then  $\mathcal E_R$ is a closed, symmetric and densely defined bilinear form on $L^2(\Om)$.  The  selfadjoint operator $(-\Delta)_R^s$ on $L^2(\Om)$ associated with $\mathcal E_R$ is given by 
\begin{equation*}
\begin{cases}
\displaystyle D((-\Delta)_R^s):=
   \Big\{ u \in L^2(\Omega):
    u_R \in W^{s,2}_{\beta,\Omega}\;\;\exists\; f\in L^2(\Omega)\mbox{ such that }\; u_R \mbox{ is a weak solution of } \eqref{RoPr} \\
    \hfill\text{ with right hand side } f\Big\},\\
    (-\Delta)_R^s u:= f.
    \end{cases}
\end{equation*}
Then $(-\Delta)_R^s$ is the realization in $L^2(\Omega)$ of $(-\Delta)^s$ with the nonlocal Robin exterior condition \eqref{NLRC}.
By \cite{Cl-Wa},  the operator $(-\Delta)_R^s$ has  compact resolvent and its eigenvalues form a non-decreasing sequence of real numbers $0< \lambda_1\le
\lambda_2\le \cdots\le \lambda_n\le \cdots$ such that $\lim_{n\to\infty}\lambda_n=\infty$. By \cite{LR-WA}, the operator $(-\Delta)_R^s$ enjoys the unique continuation property in the sense of \eqref{ucp}.
}
\end{example}

%%%%%%%%%%%%%%%%%%%%%%%%%%%%%%%%%%%%%%%%%%%%%%%%%%%%%%%
\section{Preliminary results}\label{preli}

In this section we fix some notations and give some preliminary results that will be used in the proofs of our main results. In particular we introduce the Hilfer time-fractional derivative which is a generalization of the Caputo and Riemann-Liouville time-fractional derivatives. We shall also prove the well-posedness of the system \eqref{main-EQ} and the associated adjoint system.

\subsection{Time-fractional derivatives and the  Mittag-Leffler functions}

Throughout the following, for a real number $\alpha>0$, we let
\begin{equation*}
g_\alpha(t):=
\begin{cases}
\frac{t^{\alpha-1}}{\Gamma(\alpha)}\;\;&\mbox{ if }\; t>0,\\
0&\mbox{ if }\; t\le 0.
\end{cases}
\end{equation*}

The (left) Riemann-Liouville fractional integral of order $\alpha>0$ of a locally integrable function $f:(0,\infty)\to\RR$ is defined by 
\begin{align*}
\mathbb I_t^\alpha f(t):=\frac{1}{\Gamma(\alpha)}\int_0^t (t-\tau)^{\alpha-1}f(\tau)\;d\tau=(g_\alpha\ast f)(t),\;\;t>0.
\end{align*}
We let $(\mathbb I_t^0f)(t):=f(t)$ for all $t>0.$ 
%This is equivalent to setting $g_0=\delta,$ the Dirac measure at the origin. 
We also note the important  semigroup property, $\mathbb I_t^\alpha \mathbb I_t^\beta=\mathbb I_t^{\alpha+\beta}$. 

We have the following property for power functions.

\begin{lemma}\label{int-frac}
Let  $\alpha\geq0$ and $\beta>-1$. Then,
	\begin{align*}
	\mathbb I_t^\alpha(t^\beta)=\dfrac{\Gamma(\beta+1)}{\Gamma(\alpha+\beta+1)}t^{\alpha+\beta},\qquad t>0.
	\end{align*}		
\end{lemma}

Let $X$ be a Banach space and let $f,g: [0,\infty)\to X$ be locally integrable. Using Laplace transform, we get that for every $t>0$ and $\alpha\ge 0$, 
\begin{equation}\label{Conv}
((\mathbb I_t^\alpha f)\ast g)(t)=(f\ast(\mathbb I_t^\alpha g))(t).
\end{equation}

The right Riemann-Liouville fractional integral of order $\alpha>0$ of a locally integrable function $u:\;(0,T) \to X$ is defined by
\begin{align}\label{I-R}
\mathbb I_{t,T}^\alpha u(t):=\frac{1}{\Gamma(\alpha)}\int_t^T(\tau-t)^{\alpha-1}u(\tau)\;d\tau,  \;\;t\in (0,T).
\end{align}

 Now let $0\le  \nu\le 1$, $0<\mu \le 1$, and let $u:[0,\infty) \to X$ be a locally integrable function. The (left) {\bf Hilfer time-fractional derivative of order $(\mu,\,\nu)$} is defined by
\begin{equation}\label{cfd}
\mathbb{D}_t^{\mu,\nu} u(t) := \mathbb I_t^{ \nu (1-\mu)}\frac{d}{dt}\left(\mathbb I_t^{(1-\nu)(1-\mu)}u\right)(t),\,\qquad
 t>0.
\end{equation}
We observe that if $\nu=0$, then 
$$\mathbb{D}_t^{\mu,0}u(t)=	\frac{d}{dt}\Big(g_{1-\mu}\ast u\Big)(t)$$
which is the Riemann-Liouville time-fractional derivative of order $\mu$, while for $\nu=1$, we have that
\begin{align}\label{h-to-c}
\mathbb{D}_t^{\mu,1} u(t)=\left(g_{1-\mu}\ast u^{\prime}\right)(t)=	\frac{d}{dt}\Big(g_{1-\mu}\ast (u-u(0)\Big)(t),
\end{align}
which corresponds to the Caputo time-fractional derivative of order $\mu.$ We refer to \cite[Section 2.1]{GW-B} for the justification of the second equality in \eqref{h-to-c} under appropriate conditions.

Also of interest is the right Hilfer time-fractional derivative of order $(\mu,\nu)$ ($0\le  \nu\le 1$, $0<\mu \le 1$) given by
\begin{align}\label{D-R}
\mathbb{D}_{t,T}^{\mu,\nu} u(t):=& -\mathbb I_{t,T}^{\nu(1-\mu)}\frac{d }{dt }\left(\mathbb I_{t,T}^{(1-\nu)(1-\mu)} u\right)(t).
\end{align}
We observe that, if $\mu=1$ and $u$ is differentiable, then
\begin{align*}
 \mathbb{D}_t^{1,\nu} u =\frac{du}{dt}\;\;\mbox{and }\; \displaystyle \mathbb{D}_{t,T}^{1,\nu} u = -\frac{du}{dt}.
 \end{align*}
 The right hand derivatives are introduced above  since they are needed for the integration by parts formula.
Indeed, we have the following  integration by parts formula  (see e.g. \cite{Agr,AlTo,THS10}). Let $0\le  \nu\le 1$, $0<\mu \le 1$. Then,
 \begin{align}\label{IP01}
\int_0^Tv(t)  \mathbb D_{t}^{\mu,\nu} u(t)\;dt=\int_0^T u(t)\mathbb D_{t,T}^{\mu,1-\nu} v(t)\;dt
+\left[\mathbb I_{t,T}^{(1-\nu)(1-\mu)}u(t)I_{t}^{\nu(1-\mu)}v(t)\right]_{t=0}^{t=T},
\end{align}
provided that the left and right-hand side expression makes sense.
 Special cases related to the Caputo and Riemann-Liouville fractional derivatives are easily obtained from the above formula.

The following formula will be useful (see e.g. \cite[Theorem 1]{THS10}). Let $\beta:=\mu+\nu(1-\mu)$. Then
\begin{equation}\label{hilfer-ml}
 \mathbb{D}_t^{\mu,\nu} \left[t^{\beta-1}E_{\mu,\beta}(-\omega t^\mu)\right]= -\omega t^{\beta-1}E_{\mu,\beta}(-\omega t^\mu),
\end{equation}
 where  $E_{\mu,\beta}$ is the Mittag-Leffler function defined in \eqref{mm} below. 
The proof of \eqref{hilfer-ml} is a straightforward application of the Laplace transform.

The Laplace transform of the Hilfer time-fractional derivative of a function $f$ is given by:
\begin{equation}\label{lap-hilfer}
\mathcal{L}(\mathbb{D}_t^{\mu,\nu} f)(\lambda) := \lambda^{\mu}\mathcal{L}(f)(\lambda)-\lambda^{-\nu(1-\mu)} \left(\mathbb I_t^{(1-\nu)(1-\mu)}f\right)(0^+).
\end{equation}

For more information on the Hilfer time-fractional derivative we refer to \cite{THS10} and the references therein.

The Mittag-Leffler function with two parameters is defined as follows:
\begin{align}\label{mm}
E_{\alpha, \beta}(z) := \sum_{n=0}^{\infty}\frac{z^n}{\Gamma(\alpha n + \beta)},\; \;\alpha>0,\;\beta \in\CC, \quad z \in \CC.
\end{align}
It is well-known that $E_{\alpha,\beta}(z)$ is an entire function. This is so even if we allow the parameter set to include $\rm{Re}(\alpha)>0.$
The following estimate of the Mittag-Leffler function will be useful. Let $0<\alpha<2$, $\beta\in\RR$ and $\kappa$ be such that $\frac{\alpha\pi}{2}<\kappa<\min\{\pi,\alpha\pi\}$. Then there exists a constant $C=C(\alpha,\beta,\kappa)>0$ such that
\begin{equation}\label{Est-MLF}
|E_{\alpha,\beta}(z)|\le \frac{C}{1+|z|},\;\;\;\kappa\le |\mbox{arg}(z)|\le \pi.
\end{equation}
In the literature, frequently the notation $E_\alpha=E_{\alpha, 1}$ is used. We further note that $E_{1,1}(z)=e^z$.
The Laplace transform  of the Mittag-Leffler function is given by the relation:
\begin{equation}\label{lap-ml}
\int_0^{\infty} e^{-\lambda t} t^{\alpha k + \beta-1}
E_{\alpha,\beta}^{(k)}(\pm \gamma t^{\alpha})dt = \frac{k!
\lambda^{\alpha-\beta}}{(\lambda^{\alpha} \mp \gamma)^{k+1}},
\quad \mbox{Re}(\lambda)> |\gamma|^{1/\alpha}.
\end{equation}
Here, $k\in\mathbb{N}\cup\{0\}$ and $\gamma\in\mathbb{R}.$ 
%Using \eqref{lap-ml}, we obtain for $0<\alpha<  2$:
%\begin{equation}\label{Der-ML}
%\mathbb{D}_t^\alpha E_{\alpha,1}(z t^{\alpha})= z E_{\alpha,1}(z
%t^{\alpha}), \quad t>0, z \in \CC,
%\end{equation}
%that is, for every $z\in\CC$, the function $u(t):=E_{\alpha,1}(zt^\alpha)$ is a solution of the scalar valued problem
%\begin{align*}
%\mathbb D_t^\alpha u(t)=zu(t),\;\;t> 0,\; 0<\alpha< 2,
%\end{align*}
%where $\mathbb D_t^\alpha$ denotes the Caputo time fractional derivative of order $\alpha$ (see \eqref{h-to-c}).
%Moreover, we have that for $\lambda>0$ and $m\in\NN$,
%\begin{equation}\label{Est-MLF2}
%\frac{d^m}{dt^m}E_{\alpha,1}(-\lambda t^\alpha)=-\lambda t^{\alpha-m}E_{\alpha,\alpha-m+1}(-\lambda t^\alpha),\;\;t>0.
%\end{equation}
%The proofs of \eqref{Est-MLF}, \eqref{lap-ml}, \eqref{Der-ML} and \eqref{Est-MLF2} are contained in \cite{Po99}.

 For more details on fractional derivatives, integrals and the  Mittag-Leffler functions  we refer to \cite{Agr, Ba01,Go-Ma97,Ma97,Go-Ma00,Mi-Ro,Po99} and the references therein.

 \subsection{Well-posedness of Hilfer type time-fractional evolution equations}

Throughout the rest of the paper, without any mention, we assume that the operator $A_B$ satisfies Assumption \ref{asum-A}. Moreover,
$(\varphi_n)_{n\in\NN}$ denotes the orthonormal basis of eigenfunctions of $A_B$ associated with the eigenvalues $(\lambda_n)_{n\in\NN}$.

Let $0\le  \nu\le 1$, $0<\mu \le 1$ and consider the following fractional order homogeneous evolution equation:
\begin{equation}\label{main-EQ1}
\begin{cases}
\mathbb D_t^{\mu,\nu} u+A_Bu=0\;\;&\mbox{ in }\; \Omega\times (0,T),\\
(\mathbb I_t^{(1-\nu)(1-\mu)}u)(\cdot,0)=u_0 &\mbox{ in }\;\Omega,
\end{cases}
\end{equation}
where the initial datum $u_0 \in L^2(\Omega)$ and the fractional integrals and derivatives have been defined in Subsection 3.1.

Here is our notion of solutions.

\begin{definition}
Let $T>0$. We say that a function $u\in C((0,T];V_{\frac 12})$ is  a weak solution of  \eqref{main-EQ1}, if $\mathbb D_t^{\mu,\nu} u\in C((0,T);V_{-\frac 12})$,  $\mathbb I_t^{(1-\nu)(1-\mu)}u\in C([0,T];L^2(\Omega))$,  $\mathbb I_t^{(1-\nu)(1-\mu)}u(\cdot,0)=u_0$ a.e. in $\Omega$ and 
\begin{align*}
\langle \mathbb D_t^{\mu,\nu} u(\cdot,t),\varphi\rangle_{V_{-\frac 12},V_{\frac 12}}+\mathcal E_B(u(\cdot,t),\varphi)=0,
\end{align*}
for every $\varphi\in V_{\frac 12}$ and a.e. $t\in (0,T)$.
\end{definition}

We notice that if $(1-\mu)(1-\nu)=0,$ then $\mathbb I_t^{(1-\nu)(1-\mu)}u(\cdot,0)=u(\cdot,0)=u_0.$ 

Next, we define the following operator.

\begin{definition} 
Let $0<\mu\le 1$.
Given $u\in L^2(\Omega)$ and $t\ge 0$, we let
\begin{align}\label{Ope}
S_\mu(t)u:=\sum_{n=1}^{\infty}(u,\varphi_n)_{L^2(\Om)}E_{\mu,\mu}(-\lambda_nt^\mu)\varphi_n.
\end{align}
\end{definition}

We have the following result.

\begin{lemma}\label{family-op} Let  $S_\mu(t)$ be the operator defined in \eqref{Ope}.
%\begin{align*}
%S_\mu(t)u=\sum_{n=1}^{\infty}(u,\varphi_n)E_{\mu,\mu}(-\lambda_nt^\mu)\varphi_n
%\end{align*}
Then the following assertions hold.

\begin{enumerate}
\item For any fixed $t\geq 0$, $S_\mu(t)$ is a bounded linear operator from $L^2(\Omega)$ into $L^2(\Omega)$. More precisely, there is a constant $C_1>0$ such that for every $u\in L^2(\Omega)$ and $t\ge 0$, we have
\begin{align}
\norm{S_\mu(t)u}_{L^2(\Omega)}\le C_1\norm{u}_{L^2(\Omega)}.
\end{align}

\item For every $u\in D(A_B)$, we have that $S_\mu(t)u\in D(A_B)$ for all $t\ge 0$.

\item $S_\mu(t)S_\mu(\tau)=S_\mu(\tau)S_\mu(t)$ for all $t, \tau\geq0$.

\item There is a constant $C_2>0$ such that for every $u\in L^2(\Omega)$ and $t>0$, we have 
$$\norm{\dfrac{dS_\mu(t)u}{dt}}_{L^2(\Om)}\le C_2t^{-\mu}\norm{u}_{L^2(\Omega)}.$$

\item There is a constant $C_3>0$ such that for every $t>0$ and $u\in L^2(\Omega)$, we have 
$$\norm{\mathbb I_t^{\nu(1-\mu)}(t^{\mu-1}S_\mu(t)u)}_{L^2(\Om)}\le C_3t^{-(1-\nu)(1-\mu)}\norm{u}_{L^2(\Om)}.$$
\end{enumerate}
\end{lemma}

\begin{proof}
(a) This assertion follows directly from the definition of the operator $S_\mu(t)$ given in \eqref{Ope} and the estimate of the Mittag-Leffler function given in  \eqref{Est-MLF}. 

(b) Let $u\in D(A_B)$ and $t\ge 0$. Then using \eqref{Est-MLF} again, we get that
\begin{align*}
\norm{S_\mu(t)u}^2_{D(A_B)}\le&\sum_{n=1}^\infty\abs{\lambda_n(S_\mu(t)u,\varphi_n)_{L^2(\Om)}}^2=\sum_{n=1}^\infty\abs{\lambda_n(u,\varphi_n)_{L^2(\Om)}E_{\mu,\mu}(-\lambda_nt^\alpha)}^2\le C\norm{u}^2_{D(A_B)}. 
\end{align*}
Thus, $S_\mu(t)u\in D(A_B)$. 

(c) This part is obtained by a simple calculation and using the fact that $(\varphi_n)_{n\in \mathbb{N}}$ is an orthonormal basis of $L^2(\Om)$.

(d) Let $u\in L^2(\Omega)$ and $t>0$. Since the series in \eqref{Ope} converges in $L^2(\Om)$ uniformly on compact subsets of $[0,\infty)$ (this can be easily justified), we have that
\begin{align}\label{DS-1}
\dfrac{dS_\mu(t)u}{dt}=\sum_{n=1}^\infty(u,\varphi_n)_{L^2(\Om)}\dfrac{d}{dt}\Big(E_{\mu,\mu}(-\lambda_nt^\mu)\Big)\varphi_n.
\end{align}
By  \cite[Theorem 5.1]{HMS}, the derivative of the Mittag-Leffler function is given by
 \begin{align*}
\dfrac{d}{dt}\Big[E_{\mu,\mu}(-\lambda_nt^\mu)\Big]=\frac{E_{\mu,\mu-1}(-\lambda_nt^\mu)+(1-\mu)E_{\mu,\mu}(-\lambda_nt^\mu)}{-\mu\lambda_nt^\mu}.
\end{align*}
Using  \eqref{Est-MLF} and the fact that $\lambda_n\ge \lambda_1>0$ for every $n\in\NN$,  we get that there is a constant $C>0$ such that 
\begin{align}\label{DS-2}
\abs{\dfrac{d}{dt}\Big[E_{\mu,\mu}(-\lambda_nt^\mu)\Big]}=\abs{\frac{E_{\mu,\mu-1}(-\lambda_nt^\mu)+(1-\mu)E_{\mu,\mu}(-\lambda_nt^\mu)}{-\mu\lambda_nt^\mu}}\le Ct^{-\mu}.
\end{align}
Thus, the assertion follows by combining \eqref{DS-1}-\eqref{DS-2}.
 
(e)  This part follows directly by applying Lemma \ref{int-frac}. The proof is finished.
\end{proof}

We have the following result of existence and uniqueness of weak solutions.

\begin{theorem}\label{classical}
Let $0\le  \nu\le 1$, $0<\mu \le 1$.
Then for every $u_0\in L^2(\Omega)$, the system \eqref{main-EQ1} has a unique weak solution $u$ given by
\begin{align}\label{class-sol}
 u(\cdot,t)=\sum_{n=1}^\infty (u_0,\varphi_n)_{L^2(\Om)}t^{-(1-\nu)(1-\mu)}E_{\mu,\mu+\nu(1-\mu)}(-\lambda_nt^\mu)\varphi_n,\,\, t>0,
\end{align} 
or equivalently 
\begin{align}
u(\cdot,t)=\mathbb I_t^{\nu(1-\mu)}\Big(t^{\mu-1}S_\mu(t)u_0\Big),\quad t>0.
\end{align}
\end{theorem}

\begin{proof} 
We give the main ideas of the proof.
By Lemma \ref{family-op}(e), we have that
\begin{align}\label{Est-AU}
\norm{u(\cdot,t)}_{L^2(\Omega)}= \norm{\mathbb I_t^{\nu(1-\mu)}\Big(t^{\mu-1}S_\mu(t)u_0\Big)}_{L^2(\Omega)}
\le Ct^{(1-\nu)(\mu-1)}\norm{u_0}_{L^2(\Omega)}.
\end{align}
We can also easily prove that the series in \eqref{class-sol} is convergent in $L^2(\Omega)$ uniformly in $t\in[\varepsilon, T]$, for every $0<\varepsilon<T$. Thus, we can conclude that $u\in C((0,T];L^2(\Omega))$.  Similarly, we can show that $\mathbb I_t^{(1-\nu)(1-\mu)}u\in C([0,T];L^2(\Omega))$.

Using  \eqref{lap-hilfer} and the Laplace transform, we get that the initial condition is satisfied. 

Next, it follows from \eqref{Est-MLF} that 
\begin{align*}
\norm{\mathbb{D}^{\mu,\nu}_tu(\cdot,t)}_{L^2(\Omega)}=\norm{A_Bu(\cdot,t)}_{L^2(\Omega)}\le C_2t^{\nu(1-\mu)-1}\norm{u_0}_{L^2(\Omega)}.
\end{align*}
Using the fact that the series 
\begin{align*}
\sum_{n=1}^\infty \lambda_n(u_0,\varphi_n)_{L^2(\Om)}t^{-(1-\nu)(1-\mu)}E_{\mu,\mu+\nu(1-\mu)}(-\lambda_nt^\mu)\varphi_n,
\end{align*}
converges in $L^2(\Omega)$ uniformly in $t\in[\varepsilon, T]$, for every $0<\varepsilon<T$, we can also deduce that $\mathbb{D}^{\mu,\nu}_t u\in C((0,T];L^2(\Omega))\subset C((0,T];V_{-\frac 12})$. The uniqueness of solutions is easy to verify. The proof is finished. 
\end{proof}

%In order to show the existence of weak solutions to the system \eqref{main-EQ}, we shall frequently use the following estimates which follow from  \eqref{Est-MLF}, \eqref{lap-ml}, \eqref{Est-MLF2} and some straightforward computations.

The following result will be needed (see e.g. \cite{Ke-Wa} for the proof).

\begin{lemma}\label{lem-INE}
Let $0<\alpha<2$, $T>0$ and $\lambda>0$. Then, 
\begin{align}\label{E1}
\int_0^{T}t^{\alpha-1}E_{\alpha,\alpha}(-\lambda t^\alpha)\;dt=
-\frac{1}{\lambda}\int_0^{T}\frac{d}{dt}E_{\alpha,1}(-\lambda t^\alpha)\;dt
=\frac{1}{\lambda}\Big(1-E_{\alpha,\alpha}(-\lambda T^\alpha)\Big).
\end{align}
\end{lemma}

Next, we show the existence and uniqueness of weak solutions to the system \eqref{main-EQ}.

\begin{theorem}\label{theo-weak}
Let $0\le  \nu\le 1$, $0<\mu \le 1$,  $u_0\in L^2(\Omega)$ and $f\in L^{2}((0,T);L^2(\omega))$. Then the system \eqref{main-EQ} has  a unique weak solution $u$ given by
\begin{align}\label{sol-spec}
  u(\cdot,t)=&\sum_{n=1}^\infty (u_0,\varphi_n)_{L^2(\Om)}t^{-(1-\nu)(1-\mu)}E_{\mu,\mu+\nu(1-\mu)}(-\lambda_nt^\mu)\varphi_n\\
 &+\sum_{n=1}^\infty\left(\int_0^t (f(\cdot,\tau),\varphi_n)_{L^2(\Om)}(t-\tau)^{\mu-1}E_{\mu,\mu}(-\lambda_n(t-\tau)^\mu)\;d\tau\right)\varphi_n\notag,
 \end{align}
or equivalently
 \begin{align}
u(\cdot,t)=\mathbb I_t^{\nu(1-\mu)}t^{\mu-1}S_\mu(t)u_0+\int_0^t\,(t-\tau)^{\mu-1}S_\mu(t-\tau)f(\cdot,\tau)\,d\tau.
\end{align}
% Moreover, there are some  constants $C_1,C_2>0$ such that for all $t\in (0,T)$,
% \begin{equation}\label{EST-1}
% \begin{cases}
% \|u(\cdot,t)\|_{L^2(\Omega)}\le C_1t^{(\nu-1)(1-\mu)}\|u_0\|_{L^2(\Omega)}+C_1\|f\|_{L^\infty((0,T);L^2(\Omega))},\\
% \|\mathbb D_t^{\mu,\nu} u(\cdot,t)\|_{L^2(\Omega)}^2\le C_2\left(t^{2\nu(1-\mu)-2}\|u_0\|_{L^2(\Omega)}^2+\|f\|_{L^\infty((0,T);L^2(\Omega))}^2\right).
% \end{cases}
% \end{equation}
\end{theorem}

%\MW{
\begin{proof}
The proof follows as the proof of Theorem \ref{classical}. We omit the details for brevity.
\end{proof}
%}

\subsection{Well-posedness of the associated adjoint system}

In order to investigate the controllability properties of the system \eqref{main-EQ},  we need to study the following backward system: 
\begin{equation}\label{ACP-Dual}
\begin{cases}
\mathbb D_{t,T}^{\mu,(1-\nu)} v +A_Bv=0\;\;&\mbox{ in }\; \Omega\times (0,T),\\
%Bu=0&\mbox{ on }\;\pOm\times (0,T),\\
\mathbb I_{t,T}^{\nu(1-\mu)}v(\cdot,T)=v_0 \;\;\;\; \;&\mbox{ in }\;\Omega,
\end{cases}
\end{equation}
which (by using the integration by parts formula \eqref{IP01}) can be viewed as the adjoint system associated with  \eqref{main-EQ}.

We adopt the following definition of weak solutions to the backward system \eqref{ACP-Dual}.

\begin{definition}
Let $v_0\in L^2(\Omega)$ and $T>0$.
A function $v$ is said to be a weak solution of the system \eqref{ACP-Dual}, if  the following properties hold.
\begin{itemize}
\item Regularity and final condition:
\begin{equation}\label{Dual-egu}
\begin{cases}
v\in C([0,T);V_{\frac 12}),\\
\mathbb I_{t,T}^{\nu(1-\mu)}v\in C([0,T];L^2(\Omega)), \\
\mathbb D_{t,T}^{\mu,(1-\nu)} v\in C([0,T);V_{-\frac 12}),
\end{cases}
\end{equation}
and $\mathbb I_{t,T}^{\nu(1-\mu)}v(\cdot,T)=v_0$.
\item Variational  identity: For every $\varphi\in V_{\frac 12}$ and a.e. $t\in (0,T)$, we have
\begin{align}\label{Dual-Var-I}
\langle \mathbb D_{t,T}^{\mu,(1-\nu)} v(\cdot,t),\varphi\rangle_{V_{-\frac 12},V_{\frac 12}} +\mathcal E_B(v(\cdot,t),\varphi)=0.
\end{align}
\end{itemize}
\end{definition}

%First, we show that the approximate controllability of the system \eqref{main-EQ} is equivalent to the unique continuation property for the dual system.

%\begin{theorem}\label{theo024}
%The following assertions are equivalent.
%\begin{enumerate}
%\item The system \eqref{main-EQ} is approximately controllable for any $T>0$, $\omega\subset\Omega$ an open set and for every $f\in L^\infty((0,T), L^2(\omega))$.

%\item The following unique continuation principle holds for the solution of \eqref{ACP-Dual}:
%\begin{equation}
%v(x,t)=0\;\mbox{ in }\;\omega\times (0,T)\;\Longrightarrow\; (v_0,v_1)=(0,0).
%\end{equation}
%\end{enumerate}
%\end{theorem}

Next, we show the existence and uniqueness of solutions to the backward system \eqref{ACP-Dual}.

\begin{theorem}\label{Dual-theo-weak}
Let $0\le  \nu\le 1$, $0<\mu \le 1$, and $v_0\in L^2(\Omega)$. Then the  system \eqref{ACP-Dual} has  a unique weak solution $v$ given by
 \begin{align}\label{Dual-sol-spec}
 v(\cdot,t)=&\sum_{n=1}^\infty (v_0,\varphi_n)_{L^2(\Om)}(T-t)^{-\nu(1-\mu)}E_{\mu,1-\nu(1-\mu)}(-\lambda_n(T-t)^\mu)\varphi_n.
 \end{align}
%and there   exist two  constants $C_1,C_2>0$ such that
% \begin{equation}\label{Dual-EST-1}
% \begin{cases}
% \|\mathbb I_{t,T}^{\nu(1-\mu)}v(\cdot,t)\|_{L^2(\Omega)}^2\le C_1\|v_1\|_{L^2(\Omega)}^2,\;&t\in [0,T],\\
%% \|D_{t,T}^{\alpha-1}v(t)\|_{L^2(\Omega)}^2\le C_2\left(\|v_0\|_{V_\gamma}^2+\|v_1\|_{L^2(\Omega)}^2\right)\\
% \|D_{t,T}^{\mu,1-\nu} v(\cdot,t)\|_{L^2(\Omega)}^2\le C_2(T-t)^{2(1-\nu)(1-\mu)-2}\|v_0\|_{L^2(\Omega)}^2,\;&t\in [0,T).
% \end{cases}
% \end{equation}
 Moreover, the unique weak solution $v$ can be analytically extended to the half-plane 
 $$\Sigma_T:=\{z\in\CC:\;\mbox{Re}(z)<T\}.$$
\end{theorem}

\begin{proof}
Since the representation \eqref{Dual-sol-spec} and the analytic continuation of solutions are needed in the proof of the main results, we provide more details.

Let $0\le  \nu\le 1$, $0<\mu \le 1$ and $v_0\in L^2(\Omega)$.
First, we show the uniqueness of solutions. Indeed, let $v$ be a solution of \eqref{ACP-Dual} with $v_0=0$. Taking the inner product of \eqref{ACP-Dual} with $\varphi_n$ and setting $v_n(t):=(v(t),\varphi_n)_{L^2(\Om)}$, we get that (given that $A_B$ is a selfadjoint operator)
\begin{equation}\label{O1}
\mathbb D_{t,T}^{\mu,1-\nu} v_n(t)=-\lambda_nv_n(t),\;\;\mbox{ for a.e. }\; t\in (0,T).
\end{equation}
Since $\mathbb I_{t,T}^{\nu(1-\mu)}v\in C([0,T];L^2(\Omega))$, it follows that $\mathbb I_{t,T}^{\nu(1-\mu)}v_n(t)=(\mathbb I_{t,T}^{\nu(1-\mu)}v(\cdot,t),\varphi_n)_{L^2(\Om)}\in C[0,T]$ and
\begin{align*}
\left|\mathbb I_{t,T}^{\nu(1-\mu)}v_n(t)\right|^2\le \sum_{n=1}^\infty\left|\mathbb I_{t,T}^{\nu(1-\mu)}v_n(t)\right|^2\le \left\|\mathbb I_{t,T}^{\nu(1-\mu)}v(\cdot,t)\right\|_{L^2(\Omega)}^2\to 0\;\mbox{ as }\; t\to T.
\end{align*}
This implies that
\begin{equation}\label{O2}
\mathbb I_{t,T}^{\nu(1-\mu)}v_n(T)=0.
\end{equation}
Since the fractional ordinary differential equation \eqref{O1} with the final condition \eqref{O2} has a unique solution $v_n$ given by
\begin{align*}
v_n(t)=(T-t)^{-\nu(1-\mu)}E_{\mu,1-\nu(1-\mu)}(-\lambda_n(T-t)^\mu)\mathbb I_{t,T}^{\nu(1-\mu)}v_n(T),
\end{align*}
it follows that $v_n(t)=0$ for every $n\in\NN$. Since $(\varphi_n)_{n\in\NN}$ is a complete orthonormal system in $L^2(\Omega)$, we have that $v=0$ in $\Omega\times (0,T)$ and the proof of the uniqueness is complete.

Next, we show the existence of solutions. Let  $v_{0,n}:=(v_0,\varphi_n)_{L^2(\Om)}$, $1\le n\le k$ where $k \in\NN$,  and set
\begin{align*}
v_k(x,t):=&\sum_{n=1}^k v_{0,n}(T-t)^{-\nu(1-\mu)}E_{\mu,1-\nu(1-\mu)}(-\lambda_n(T-t)^\mu)\varphi_n(x).
\end{align*}

(a) Let $v$ be given by \eqref{Dual-sol-spec}. We claim that  $\mathbb I_{t,T}^{\nu(1-\mu)}v\in C([0,T];L^2(\Omega))$. Integrating termwise, we have that (see \cite{THS10})
\begin{align}\label{V1}
\mathbb I_{t,T}^{\nu(1-\mu)}v_k(\cdot,t)=&\sum_{n=1}^k v_{0,n}E_{\mu,1}(-\lambda_n(T-t)^\mu)\varphi_n.
\end{align}
Using \eqref{Est-MLF} and the estimates in Lemma \ref{lem-INE}, we have that for every $t\in [0,T]$ and $m$, $k\in\NN$ with $m>k$,
\begin{align*}
\|\mathbb I_{t,T}^{\nu(1-\mu)}v_k(t)-\mathbb I_{t,T}^{\nu(1-\mu)}v_m(t)\|_{L^2(\Omega)}^2=&\sum_{n=k+1}^m |v_{0,n}E_{\mu,1}(-\lambda_n(T-t)^\mu)|^2\\
\le &C^2\sum_{n=k+1}^m|v_{0,n}|^2\rightarrow 0\;\mbox{ as }\; k,m\to\infty.
\end{align*}
We have shown that  
\begin{align*}
&\sum_{n=1}^\infty v_{0,n}E_{\mu,1}(-\lambda_n(T-t)^\mu)\varphi_n \to    \mathbb I_{t,T}^{\nu(1-\mu)}v(\cdot,t)\;\mbox{ in }\;L^2(\Omega),
\end{align*}
and that the convergence is uniform in $t\in [0,T]$. Hence, $\mathbb I_{t,T}^{\nu(1-\mu)}v\in C([0,T];L^2(\Omega))$. Using \eqref{Est-MLF} and Lemma \ref{lem-INE} again, we get that there is a constant $C>0$ such that for every $t\in (0,T)$ we have
\begin{align}
 \|\mathbb I_{t,T}^{\nu(1-\mu)}v(\cdot,t)\|_{L^2(\Omega)}^2\le C^2\|v_0\|_{L^2(\Omega)}^2.
\end{align}

(b) We prove that $\mathbb D_{t,T}^{\mu,1-\nu} v\in C([0,T);L^2(\Omega))$. Since $\mathbb D_{t,T}^\alpha v(\cdot,t)=-A_Bv(\cdot,t)$, we have that
\begin{align}
\|D_{t,T}^{\mu,1-\nu}v(\cdot,t)\|_{L^2(\Omega)}^2\le &\sum_{n=1}^\infty |v_{0,n}|^2|\lambda_n(T-t)^{-\nu(1-\mu)}E_{\mu,1-\nu(1-\mu)}(-\lambda_n(T-t)^\mu)|^2\notag\\
\le &C(T-t)^{2(\nu\mu-\nu-\mu)}\|v_0\|_{L^2(\Omega)}^2.
\end{align}
Proceeding as above we can deduce that $\mathbb D_{t,T}^{\mu,1-\nu} v\in C([0,T);L^2(\Omega))$.

(c) It follows from \eqref{V1} that
\begin{align*}
\mathbb I_{t,T}^{\nu(1-\mu)}v(\cdot,T)=\sum_{n=1}^\infty v_{0,n}\varphi_n=v_0.
\end{align*}

(d) Finally, since $E_{\mu,1-\nu(1-\mu)}(-\lambda_nz)$ is an entire function, it follows that the function
\begin{align*}
(T-t)^{-\nu(1-\mu)}E_{\mu,1-\nu(1-\mu)}(-\lambda_n(T-t)^\mu)
\end{align*}
 can be analytically extended to the half-plane $\Sigma_T$. This implies that the function
 \begin{align*}
 \sum_{n=1}^k v_{0,n}(T-z)^{-\nu(1-\mu)}E_{\mu,1-\nu(1-\mu)}(-\lambda_n(T-z)^\mu)\varphi_n
\end{align*}
is analytic in $\Sigma_T$. Let $\delta>0$ be fixed but otherwise arbitrary. Let $z\in\CC$ satisfy $\mbox{Re}(z)\le T-\delta$. Then, using Lemma \ref{lem-INE}, we get that
 \begin{align*}
& \left\Vert\sum_{n=k+1}^\infty v_{0,n}(T-z)^{-\nu(1-\mu)}E_{\mu,1-\nu(1-\mu)}(-\lambda_n(T-z)^\mu)\varphi_n\right\Vert_{L^2(\Omega)}^2\\
  \le &C\sum_{n=k+1}^\infty |v_{0,n}|^2|T-z|^{2(\nu\mu-\nu)}\left(\frac{1}{1+\lambda_n|T-z|^{\mu}}\right)^2\\
  \le &C\delta^{2(\nu\mu-\nu)}\sum_{n=N+1}^\infty |v_{0,n}|^2 \to 0\;\mbox{ as }\; k\to\infty.
 \end{align*}
 We have shown that
 \begin{align*}
 v(\cdot,z):=&\sum_{n=1}^\infty (v_0,\varphi_n)(T-z)^{\nu\mu-\nu}E_{\mu,1-\nu(1-\mu)}(-\lambda_n(T-z)^\mu)\varphi_n
 \end{align*}
 is uniformly convergent in any compact subset of $\Sigma_T$. Hence, $v$ is also analytic in $\Sigma_T$.
The proof of the theorem is finished.
\end{proof}

\begin{remark}\label{rem-dual-sol}
{\em 
We notice that the solution $v$ of the backward system \eqref{ACP-Dual} satisfies the following additional regularity: There is a constant $C>0$ such that
\begin{equation}\label{sol-v}
\|v(\cdot,t)\|_{L^2(\Omega)}^2\le C(T-t)^{-2\nu(1-\mu)}\|v_0\|_{L^2(\Omega)}^2.
\end{equation}
Using \eqref{sol-v} we can deduce that $v\in C([0,T);L^2(\Omega))\cap L^1((0,T);L^2(\Omega))$.
}
\end{remark}

Next, we show that under the assumption that $A_B$ has the unique continuation property, the adjoint system \eqref{ACP-Dual} satisfies the  unique continuation principle.

\begin{proposition}\label{pro-uni-con}
Let $0\le  \nu\le 1$, $0<\mu \le 1$.
Let $v_0\in L^2(\Omega)$ and let $\omega\subset\Omega$ be an arbitrary non-empty open set. Assume that $A_B$ has the unique continuation property in the sense of \eqref{ucp}.
Let $v$ be the unique weak solution of \eqref{ACP-Dual}. If $v=0$ on $\omega\times (0,T)$, then $v=0$ on $\Omega\times (0,T)$.
\end{proposition}

\begin{proof}
Let $v_0\in L^2(\Omega)$ and let $\omega\subset\Omega$ be an arbitrary non-empty open set. Let $v$ be the unique weak solution of \eqref{ACP-Dual} and assume that $v=0$ in $\omega\times (0,T)$.
Since $v=0$ in $\omega\times (0,T)$ and $v:\; [0,T)\to L^2(\Omega)$ can be analytically extended to the half-plane $\Sigma_T$ (by Theorem \ref{Dual-theo-weak}), it follows that for a.e. $x\in\omega$ and $t\in (-\infty,T)$, we have
\begin{align}\label{Exp-v}
v(x,t)=\sum_{n=1}^\infty (v_0,\varphi_n)_{L^2(\Om)}(T-t)^{-\nu(1-\mu)}E_{\mu,1-\nu(1-\mu)}(-\lambda_n(T-t)^\mu)\varphi_n(x)=0.
\end{align}
Let $(\lambda_k)_{k\in\NN}$ be the set of all eigenvalues of the operator $A_B$ and let $(\psi_{k_j})_{1\le j\le m_k}$ be an orthonormal basis for $\mbox{Ker}(\lambda_k-A_B)$, where $m_k$ is the multiplicity of $\lambda_k$. Then \eqref{Exp-v} can be rewritten as
\begin{align}\label{e156}
v(x,t)=&\sum_{k=1}^\infty \left(\sum_{j=1}^{m_k}(v_0,\psi_{k_j})_{L^2(\Om)}\psi_{k_j}(x)\right)(T-t)^{-\nu(1-\mu)}E_{\mu,1-\nu(1-\mu)}(-\lambda_k(T-t)^\mu)\notag\\
  &=0,\;\;\qquad\qquad x\in\omega,\;t\in (-\infty,T).
\end{align}
Let $z\in\CC$ with $\mbox{Re}(z):=\eta>0$ and let $M\in\NN$. Since the system $\{\psi_{k_j}\}$, for $1\le j\le m_k$, $1\le k\le M$ is orthonormal, we have that there is a constant $C>0$ such that
\begin{align*}
&\left\Vert \sum_{k=1}^M \left(\sum_{j=1}^{m_k}(v_0,\psi_{k_j})_{L^2(\Om)}\psi_{k_j}(x)\right)e^{z(t-T)}(T-t)^{-\nu(1-\mu)}E_{\mu,1-\nu(1-\mu)}(-\lambda_k(T-t)^\mu)\right\Vert_{L^2(\Omega)}^2\\
\le &\sum_{k=1}^\infty \left(\sum_{j=1}^{m_k}|(v_0,\psi_{k_j})_{L^2(\Om)}|^2\right)e^{2\eta(t-T)}|(T-t)^{-\nu(1-\mu)}E_{\mu,1-\nu(1-\mu)}(-\lambda_k(T-t)^\mu)|^2\\
\le & Ce^{2\eta(t-T)}(T-t)^{2\nu\mu-2\nu}\|v_0\|_{L^2(\Omega)}^2.
\end{align*}
Letting
\begin{align*}
w_M(\cdot,t):=&\sum_{k=1}^M \left(\sum_{j=1}^{m_k}(v_0,\psi_{k_j})_{L^2(\Om)}\psi_{k_j}(x)\right)e^{z(t-T)}(T-t)^{-\nu(1-\mu)}E_{\mu,1-\nu(1-\mu)}(-\lambda_k(T-t)^\mu),
\end{align*}
we have shown that
\begin{align}\label{nporm-2}
&\|w_M(\cdot,t)\|_{L^2(\Omega)}
\le Ce^{\eta(t-T)}(T-t)^{-\nu(1-\mu)}\|v_0\|_{L^2(\Omega)}.
\end{align}

The right hand side of \eqref{nporm-2} is integrable over $t\in (-\infty,T)$. More precisely, we have that
\begin{align*}
\int_{-\infty}^Te^{\eta(t-T)}\left[(T-t)^{-\nu(1-\mu)}\|v_0\|_{L^2(\Omega)}\right]\;dt
=&\int_0^\infty e^{-\tau}\frac{\tau^{\mu\nu-\nu+1-1}}{\eta^{1+\nu\mu-\nu}}\;d\tau |v_0\|_{L^2(\Omega)}\\
=&\frac{\Gamma(1-\nu(1-\mu))}{\eta^{1-\nu(1-\mu)}}\|v_0\|_{L^2(\Omega)}.
\end{align*}
By the Lebesgue dominated convergence theorem, we can deduce that
\begin{align}\label{e158}
&\int_{-\infty}^Te^{z(t-T)}\left[\sum_{k=1}^\infty \left(\sum_{j=1}^{m_k}(v_0,\psi_{k_j})_{L^2(\Om)}\psi_{k_j}(x)\right)(T-t)^{-\nu(1-\mu)}E_{\mu,1-\nu(1-\mu)}(-\lambda_k(T-t)^\mu)\right]dt\notag\\
=&\sum_{k=1}^\infty \sum_{j=1}^{m_k}v_{0,k_j}\psi_{k_j}(x)\left(\int_{-\infty}^Te^{z(t-T)}(T-t)^{-\nu(1-\mu)}E_{\mu,1-\nu(1-\mu)}(-\lambda_k(T-t)^\mu)dt\right)\notag\\
=&\sum_{k=1}^\infty \sum_{j=1}^{m_k}\left(\frac{(v_0,\psi_{k_j})_{L^2(\Om)}z^{-(1-\nu)(1-\mu)}}{z^{\mu}+\lambda_k}\right)\psi_{k_j}(x),\;\;x\in\Omega,\;\;\mbox{Re}(z)>0.
\end{align}
To arrive at \eqref{e158}, we have used the fact that
\begin{align*}
&\int_{-\infty}^Te^{z(t-T)}(T-t)^{-\nu(1-\mu)}E_{\mu,\mu-\nu(1-\mu)}(-\lambda_k(T-t)^\mu)dt\\
=&\int_0^\infty e^{-z\tau}\tau^{1-\nu(1-\mu)-1}E_{\mu,1-\nu(1-\mu)}(-\lambda_k\tau^\mu)\;d\tau=\frac{z^{-(1-\nu)(1-\mu)}}{z^\mu+\lambda_k}
\end{align*}
These identities follow from a simple change of variable and \eqref{lap-ml}.

It follows from \eqref{e156} and \eqref{e158} that
\begin{align*}
\sum_{k=1}^\infty \sum_{j=1}^{m_k}\left(\frac{(v_0,\psi_{k_j})_{L^2(\Om)}z^{-(1-\nu)(1-\mu)}}{z^{\mu}+\lambda_k}\right)\psi_{k_j}(x)=0,\;\;x\in\omega,\;\mbox{Re}(z)>0.
\end{align*}
Letting $\eta:=z^\mu$, we have shown that
\begin{align}\label{e416}
\sum_{k=1}^\infty \sum_{j=1}^{m_k}\left(\frac{(v_0,\psi_{k_j})_{L^2(\Om)}\eta^{\frac{-(1-\nu)(1-\mu)}{\mu}}}{\eta+\lambda_k}\right)\psi_{k_j}(x)=0,\;\;x\in\omega,\;\mbox{Re}(\eta)>0.
\end{align}
Using the analytic continuation in $\eta$,  we have that the identity  \eqref{e416} holds for every $\eta\in\CC\setminus\{-\lambda_k\}_{k\in\NN}$. Taking a suitable small circle about $-\lambda_l$ and not passing through nor encircling $\{-\lambda_k\}_{k\ne l}$ and integrating \eqref{e416} on that circle we get that
\begin{align*}
w_l:= \sum_{j=1}^{m_l}\left[(v_0,\psi_{l_j})(-\lambda_l)^{-\frac{(1-\nu)(1-\mu)}{\mu}}\right]\psi_{l_j}(x) =0,\;\;x\in\omega.
\end{align*}
Since $(A_B-\lambda_l)w_l=0$ in $\Omega$, $w_l=0$ in $\omega$,  and by assumption $A_B$ has the unique continuation property in the sense of \eqref{ucp}, it follows that $w_l=0$ in $\Omega$ for every $l$. Since $(\psi_{l_j})_{1\le j\le m_k}$ is linearly independent in $L^2(\Omega)$, we get that $(v_0,\psi_{l_j})_{L^2(\Om)}(-\lambda_l)^{\beta}=0$ for $1\le j\le m_k$, $k\in\NN$ and $\beta=-\frac{(1-\nu)(1-\mu)}{\mu}$. This implies that
\begin{align}\label{zero}
0=&(-\lambda_l)^\beta(v_0,\psi_{l_j})_{L^2(\Om)}\notag\\
=&\lambda_l^\beta\left[\cos(\beta\pi)+i\sin(\beta\pi)\right](v_0,\psi_{l_j})_{L^2(\Om)}\notag\\
=&\lambda_l^\beta\cos(\beta\pi)(v_0,\psi_{l_j})_{L^2(\Om)}+i\lambda_l^\beta\sin(\beta\pi)(v_0,\psi_{l_j})_{L^2(\Om)}.
\end{align}
It follows from \eqref{zero} that $(v_0,\psi_{l_j})_{L^2(\Om)}=0$ for $1\le j\le m_k$. Hence, $v=0$ in $\Omega\times (0,T)$. The proof is finished.
\end{proof}

\section{Proof of the main results}\label{prof-ma-re}

In this section we give the proof of the main results stated in Section \ref{main-res}. Before proceeding with the proof, we show in the following remark that to study the approximate controllability or the mean approximate controllability of the system \eqref{main-EQ}, it suffices to consider the case $u_0=0$.

\begin{remark} 
{\em
Consider the following two systems:
\begin{align}\label{EQ-0}
\begin{cases}
\mathbb D_t^{\mu,\nu} v+A_Bv=f|_{\omega\times(0,T)}\;\;&\mbox{ in }\; \Omega\times (0,T),\\
%Bu=0&\mbox{ on }\;\pOm\times (0,T),\\
(\mathbb I_t^{(1-\nu)(1-\mu)}v)(\cdot,0)=0 &\mbox{ in }\;\Omega,
\end{cases}
\end{align}
and
\begin{align}\label{EQ-U0}
\begin{cases}
\mathbb D_t^{\mu,\nu} w+A_Bw=0\;\;&\mbox{ in }\; \Omega\times (0,T),\\
%Bu=0&\mbox{ on }\;\pOm\times (0,T),\\
(\mathbb I_t^{(1-\nu)(1-\mu)}w)(\cdot,0)=u_0 &\mbox{ in }\;\Omega.
\end{cases}
\end{align}

Given $u_0\in L^2(\Om)$, let  $w$ be the weak solution of  \eqref{EQ-U0}. Assume that the system \eqref{EQ-0} is approximately controllable and let $u_1\in L^2(\Omega)$. Then, for every $\varepsilon>0$, there exists a control function $f\in L^2(\omega\times(0,T))$ such that the corresponding unique weak solution $v$ of  \eqref{EQ-0} satisfies
\begin{align}\label{control}
\|v(\cdot,T)-(u_1-w(\cdot,T))\|_{L^2(\Omega)}\le\varepsilon.
\end{align}
By definition, we have that the function $v+w$ solves the system \eqref{main-EQ}, and it follows from \eqref{control} that  
\begin{align*}
\|(v+w)(\cdot,T)-u_1\|_{L^2(\Omega)}\le\varepsilon.
\end{align*}
Hence, \eqref{main-EQ} is approximately controllable.  The case of the  mean approximate controllability follows similarly. 
}
\end{remark}

%\subsection{Proof the first main result}

%In this section we give the the proof of the theorem \ref{main-Theo} stated in Section \ref{main-res}. 

%, for every $\varepsilon>0$ and $u_1\in L^2(\Omega)$, there exists a control function $f\in L^2(\omega\times(0,T))$ such that the corresponding unique weak solution $u$ of  \eqref{main-EQ} with initial datum $u_0=0$ satisfies
%\begin{equation}\label{contr-eq}
%\|u(\cdot,T)-u_1)\|_{L^2(\Omega)}\le \varepsilon.
%\end{equation}
%
%Now we are ready to give the proof of the first main result.

\begin{proof}[\bf Proof of Theorem \ref{main-Theo}] 
Let $0\le\nu< 1$ and $0<\mu<1$.
We recall that the system \eqref{main-EQ} is approximately controllable if $\mathcal R(u_0,T)$, with $u_0=0$, is dense in $L^2(\Omega)$. Since  $D(A_B)$ (the domain of $A_B$) is dense in $L^2(\Omega)$, it is sufficient to show that
\begin{equation}\label{inclu}
D(A_B)\subseteq \Big\{u(\cdot,T):\; f\in L^2(\omega\times(0,T))\Big\}. 
\end{equation}
Indeed, assume that $u_0=0$ and let $\phi\in D(A_B)$. We set
\begin{align*}
\psi(\cdot,t):=\frac{\Gamma(\mu)^2(T-t)^{1-\mu}}{T}\left[S_\mu(T-t)-2t\dfrac{d}{dt}S_\mu(T-t)\right]\phi.
\end{align*}

We claim that $\psi\in L^2((0,T);L^2(\Om))$. Using Lemma \ref{family-op}, we get that there is a constant $C>0$ such that for every $t\in (0,T)$ we have
\begin{align}\label{jj}
\norm{\psi(\cdot,t)}^2_{L^2(\Om)}\le& \left(2\frac{\Gamma(\mu)^2(T-t)^{1-\mu}}{T}\right)^2\left[\norm{S_\mu(T-t)\phi}^2_{L^2(\Om)}+(2t)^2\norm{\dfrac{d}{dt}S_\mu(T-t)\phi}^2_{L^2(\Om)}\right]\notag\\
 \le&\left(2\frac{\Gamma(\mu)^2(T-t)^{1-\mu}}{T}\right)^2 (C^2+(2t^{1-\mu}C)^2)\norm{\phi}_{L^2(\Om)}^2.
\end{align}
Integrating \eqref{jj} over $(0,T)$ we get that
\begin{align*}
&\int_0^T\norm{\psi(\cdot,t)}^2_{L^2(\Om)}\;dt&\\
\le &\left(2\frac{\Gamma(\mu)^2\norm{\phi}_{L^2(\Om)}}{T}\right)^2\left[C^2\int_0^T(T-t)^{2(1-\mu)}\,dt+(2C)^2\int_0^Tt^{2(1-\mu)}(T-t)^{2(1-\mu)}\,dt\right]\\ 
&\le \left(2\frac{\Gamma(\mu)^2\norm{\phi}}{T}\right)^2\left[C^2\dfrac{T^{\alpha+1}}{\alpha+1}+(2C)^2T^{2\alpha+1}\dfrac{\Gamma(\alpha+1)}{\Gamma(2\alpha+2)}\right]<\infty,
\end{align*}
where we have set $\alpha:=2(1-\mu)$, and the claim is proved.

By Theorem \ref{theo-weak}, the system \eqref{main-EQ} with right hand side $\psi$ has a unique weak solution $u$ given by
\begin{align*}
u(\cdot,t)=\int_0^t(t-\tau)^{\mu-1}S_\mu(t-\tau)\psi(\cdot,\tau)\,d\tau.
\end{align*} 

We claim $u(\cdot,T)=\phi$. Using the properties contained in Lemma \ref{family-op}(e), we get that
\begin{align*}
u(\cdot,T)&=\int_0^T(T-\tau)^{\mu-1}S_\mu(T-\tau)\psi(\cdot,\tau)\,d\tau\\
&=\dfrac{\Gamma(\mu)^2}{T}\int_0^T\left(S^2_\mu(T-\tau)\phi-2\tau S_\mu(T-\tau)\dfrac{d}{dt}S_\mu(T-\tau)\phi\right)d\tau\\ 
&=\dfrac{\Gamma(\mu)^2}{T}\left(\int_0^T S^2_\mu(T-\tau)\phi\, d\tau-\int_0^T2\tau S_\mu(T-\tau)\dfrac{d}{dt}S_\mu(T-\tau)\phi\, d\tau\right)\\
 &= \dfrac{\Gamma(\mu)^2}{T}\left(\int_0^T S^2_\mu(T-\tau)\phi\, d\tau+\int_0^T\tau \dfrac{d}{d\tau} S^2_\mu(T-\tau)\phi\, d\tau\right)\\
 &=\dfrac{\Gamma(\mu)^2}{T}\left(\Big[\tau S^2_\mu(T-\tau)\phi\Big]_{\tau=0}^{\tau=T} +\int_0^T S^2_\mu(T-\tau)\phi\, d\tau-\int_0^T S^2_\mu(T-\tau)\phi\, d\tau\right)\\&=\dfrac{\Gamma(\mu)^2}{T}(TS^2_\mu(0)\phi)=\phi,
\end{align*}
where we also used that $S_\mu(0)=\frac{1}{\Gamma(\mu)}$. We have shown \eqref{inclu} and the proof is finished.
\end{proof}

%\subsection{Proof the second main result}
%Here we prove the result concerning the mean approximate controllability properties.

\begin{proof}[\bf Proof of Theorem \ref{second-Theo}]
$0\le\nu\le 1$ and $0<\mu\le 1$.
Assume that the operator $A_B$ has the unique continuation property in the sense of \eqref{ucp}.
Let $u$ be the unique weak solution of \eqref{main-EQ} with $u_0=0$ and $v$ the unique weak solution of the adjoint system \eqref{ACP-Dual} with $v_0\in L^2(\Omega)$. 
%It follows from Remark \ref{rem-dual-sol} that $v\in L^1((0,T);L^2(\Omega))$. 
Then, integrating by parts, we get that (by using \eqref{IP01} and the fact that $(A_Bu,v)_{L^2(\Om)}=(u,A_Bu)_{L^2(\Om)}$),
\begin{align}\label{int-del}
0=&\int_0^{T}\int_{\Omega}\left(\mathbb D_t^{\mu,\nu} u+A_Bu-f\right)v\;dxdt\notag\\
=&\int_0^{T}\int_{\Omega}v\mathbb D_t^{\mu,\nu} u\;dxdt+\int_0^{T}\int_{\Omega}vA_Bu\;dxdt -\int_0^{T}\int_{\omega}fv\;dxdt\notag\\
=&\int_0^{T}\int_{\Omega}u \mathbb D_{t,T}^{\mu,(1-\nu)} v\;dxdt+\int_{\Omega}\left[\mathbb I_t^{(1-\nu)(1-\mu)}u(x,T)\mathbb I_{t,T}^{\nu(1-\mu)}v(x,T)\right]\;dx\notag\\
&+\int_0^{T}\int_{\Omega}uA_Bv\;dxdt-\int_0^{T}\int_{\omega}fv\;dxdt\notag\\
=&\int_0^{T}\int_{\Omega}\left(\mathbb D_{t,T}^{\mu,(1-\nu)} v+A_Bv\right)u\;dxdt\notag\\
&+\int_{\Omega}\left[\mathbb I_t^{(1-\nu)(1-\mu)}u(x,T)\mathbb I_{t,T}^{\nu(1-\mu)}v(x,T)\right]\;dx-\int_0^{T}\int_{\omega}fv\;dxdt\notag\\
=&\int_{\Omega}\left[\mathbb I_t^{(1-\nu)(1-\mu)}u(x,T)\mathbb I_{t,T}^{\nu(1-\mu)}v(x,T)\right]\;dx-\int_0^{T}\int_{\omega}fv\;dxdt.
\end{align}
We  have shown  that
\begin{align}\label{eq41}
\int_{\Omega}\left[\mathbb I_t^{(1-\nu)(1-\mu)}u(x,T)v_0\right]\;dx=\int_0^T\int_{\omega}fv\;dxdt.
\end{align}
To prove that the set 
$$\left\{\mathbb I_t^{(1-\nu)(1-\mu)}u(\cdot,T):\; f\in L^2(\omega\times (0,T))\right\}$$ 
is dense in $L^2(\Omega)$, we have to show that if $v_0\in  L^2(\Omega)$ is such that
\begin{align}\label{eq42}
\int_{\Omega}\left[\mathbb I_t^{(1-\nu)(1-\mu)}u(x,T)v_0(x)\right]\;dx=0
%\int_0^T\int_{\omega}fv\,dx dt
\end{align}
for every $f\in L^2(\omega\times (0,T))$, then $v_0=0$. Indeed, let $v_0$ satisfy \eqref{eq42}. It follows from \eqref{eq41} and \eqref{eq42} that
\begin{align*}
\int_0^T\int_{\omega}fv\;dxdt=0
\end{align*}
for every $f\in L^2(\omega\times (0,T))$. By the fundamental lemma of the calculus of variations, we have that
\begin{align*}
v=0 \;\mbox{ in }\; \omega\times (0,T).
\end{align*}
It follows from Proposition \ref{pro-uni-con} that
\begin{align*}
v=0  \;\mbox{ in }\; \Omega\times (0,T).
\end{align*}
Since the solution $v$ of \eqref{ACP-Dual} is unique, it follows that $v_0=0$ on $\Omega$. The proof of the theorem is finished.
\end{proof}

%\section{mm}

We conclude the paper with the following remark.

\begin{remark}
{\em We mention the following facts. Let $(1-\nu)(1-\mu)\ne 0$.
\begin{enumerate}
\item Consider the following mapping:
\begin{align*}
F: L^2(\omega\times(0,T))\to L^2(\Omega),\; f\mapsto \mathbb I_t^{(1-\nu)(1-\mu}u(\cdot,T),
\end{align*}
where $u$ is the unique weak solution of \eqref{main-EQ} associated to $u_0=0$. Then it is easy to see that the system \eqref{main-EQ} is mean approximately controllable in time $T>0$ if and only if the range of $F$, that is, $\mbox{Ran}(F)$ is dense in $L^2(\Omega)$. This is equivalent to $\mbox{Ker}(F^\star)=\{0\}$, where $F^\star$ is the adjoint of $F$. It follows from the proof of Theorem \ref{second-Theo} (more precisely from \eqref{eq41}) that $F^\star$ is the mapping given by
\begin{align*}
F^\star: L^2(\Omega)\to L^2(\omega\times(0,T)),\;v_0\mapsto v\big|_{\omega\times(0,T)},
\end{align*}
where $v$ is the unique solution of the adjoint system \eqref{ACP-Dual}. Again $\mbox{Ker}(F^\star)=\{0\}$ is the unique continuation principle, namely,
\begin{align*}
(v\;\mbox{ solution of }\; \eqref{ACP-Dual},\;  v\big|_{\omega\times(0,T)}=0) \Longrightarrow \;v_0=0\;\mbox{ in }\;\Omega.
\end{align*}

\item Now consider the mapping
\begin{align*}
G: L^2(\omega\times(0,T))\to L^2(\Omega),\; f\mapsto u(\cdot,T),
\end{align*}
where $u$ is the unique weak solution of \eqref{main-EQ} associated with $u_0=0$. As above, the system \eqref{main-EQ} is approximately controllable in time $T>0$ if and only if the range of $G$, that is, $\mbox{Ran}(G)$ is dense in $L^2(\Omega)$, and this is equivalent to $\mbox{Ker}(G^\star)=\{0\}$, where $G^\star$ is the adjoint of $G$. 

Next, we compute $G^\star$. Indeed, let $f\in L^2(\omega\times(0,T))$ and $\psi\in L^2(\Omega)$. Then,
\begin{align*}
(G(f),\psi)_{L^2(\Omega)}:=\int_\Omega (Gf)(x,t)\psi(x)\,dx,
\end{align*}
where
\begin{align*}
(Gf)(x, T):=\sum_{n=1}^\infty \left( \int_0^T(f(\cdot,t),\varphi_n(\cdot))_{L^2(\Omega)}(T-t)^{\mu-1}E_{\mu,\mu}(-\lambda_n(T-t)^\mu)dt\right) \varphi_n(x).
\end{align*}
Now we have that
\begin{align*}
(G(f),\psi)_{L^2(\Omega)}=&\int_\Omega\left( \sum_{n=1}^\infty \left( \int_0^T(f(\cdot,t),\varphi_n(\cdot))_{L^2(\Omega)}(T-t)^{\mu-1}E_{\mu,\mu}(-\lambda_n(T-t)^\mu)dt\right) \varphi_n(x)\right) \psi(x)\,dx\\ 
=& \sum_{n=1}^\infty \int_0^T \left( \int_\Omega (f(\cdot,t),\varphi_n(\cdot))_{L^2(\Omega)}\varphi_n(x)\psi(x)\,dx\right) (T-t)^{\mu-1}E_{\mu,\mu}(-\lambda_n(T-t)^\mu)dt\\
=&\sum_{n=1}^\infty\int_0^T\left( \int_\Omega\int_\Omega\,f(y,t)\varphi_n(y)\varphi_n(x)\psi(x)\,dydx\right)(T-t)^{\mu-1}E_{\mu,\mu}(-\lambda_n(T-t)^\mu)dt.
\end{align*}
Using \eqref{sol-spec} in Theorem \ref{theo-weak} and applying Fubini's theorem we get that for all $\psi\in L^2(\Omega)$, $0\le t<T$ and for a.e $x\in \Omega$,
\begin{align*}
(G(f),\psi)_{L^2(\Omega)}=&\sum_{n=1}^\infty\int_0^T\left( \int_\Omega(\varphi_n(\cdot),\Psi(\cdot))_{L^2(\Omega)}f(y,t)\varphi_n(y)\,dy\right)(T-t)^{\mu-1}E_{\mu,\mu}(-\lambda_n(T-t)^\mu)dt\\
&=\int_\Omega\int_0^T\left( \sum_{n=1}^\infty (\varphi_n(\cdot),\psi(\cdot))_{L^2(\Omega)}(T-t)^{\mu-1}E_{\mu,\mu}(-\lambda_n(T-t)^\mu)\right) \varphi_n(y)f(y,t)\,dtdy\\
&=(f,(G^\ast\psi))_{L^2(\Omega)}
\end{align*}
We have shown that for all $\psi\in L^2(\Omega)$, $0\le t<T$ and for a.e $x\in \Omega$,
\begin{align}\label{gg}
(G^\ast\psi)(x, t)=\sum_{n=1}^\infty (\varphi_n,\psi)_{L^2(\Omega)}(T-t)^{\mu-1}E_{\mu,\mu}(-\lambda_n(T-t)^\mu)\varphi(x).
\end{align}
We can see from \eqref{gg} that $\mbox{Ker}(G^\star)$ is not related to the adjoint system \eqref{ACP-Dual}. For that reason the approximate controllability of the system \eqref{main-EQ} (in the case $0\le \nu<1$ and $0<\mu<1$)  is not directly related to the unique continuation principle for the adjoint system \eqref{ACP-Dual}.
\end{enumerate}
}
\end{remark}

%\begin{proof}[\bf Proof of Corollary \ref{cor}]
%Assume that the system \eqref{main-EQ} is mean approximately controllable.
%
%(a) Let $\varepsilon>0$ and $u_1\in L^2(\Omega)$. Then
%\begin{align*}
%\|u(\cdot,T)-u_1\|_{L^2(\Omega)}=&\|u(\cdot,T)-\mathbb I_t^{(1-\nu)(1-\mu)}u(\cdot,T)+ \mathbb I_t^{(1-\nu)(1-\mu)}u(\cdot,T)-u_1\|_{L^2(\Omega)}\\
%\le& \|u(\cdot,T)-\mathbb I_t^{(1-\nu)(1-\mu)}u(\cdot,T)\|_{L^2(\Om)}+\| \mathbb I_t^{(1-\nu)(1-\mu)}u(\cdot,T)-u_1\|_{L^2(\Omega)}
%\end{align*}
%The estimate \eqref{u-m} follows directly from \eqref{e23} and \eqref{e24} by using the triangle inequality.
%\end{proof}

\bibliographystyle{plain}
\bibliography{biblio}
%%%%%%%%%%%%%%%%%%%%%%%%%%%%%%%%%%%%%%%%%%%%%%%%%%%%%%%%

%%%%%%%%%%%%%%%%%%%%%%%%%%%%%%%%%%%%%%%%%%%%%%%%%%%%%%%
\end{document}